\documentclass[10pt]{article}

\usepackage{amsfonts}

\usepackage{amsmath}
\usepackage{verbatim}
\usepackage{amsmath}
\usepackage{graphicx}
\usepackage[active]{srcltx}
\usepackage[
bookmarks=true,         
bookmarksnumbered=true, 
colorlinks=true,
pdfstartview=FitV,
linkcolor=blue, citecolor=blue, urlcolor=blue]{hyperref}

\newtheorem{theorem}{Theorem}[section]

\newtheorem{lemma}[theorem]{Lemma}

\newtheorem{problem}[theorem]{Problem}
\newtheorem{proposition}[theorem]{Proposition}
\newtheorem{remark}[theorem]{Remark}

\newenvironment{proof}[1][Proof]{\textbf{#1.} }{\ \rule{0.5em}{0.5em} \par }

\newtheorem{prop}[theorem]{Proposition}
\def\RR{\mathbb{R}}
\def\EE{\mathbb{E}}
\def\DD{\mathbb{D}}

\def\cD{{  D}}

\def\cF  {{\cal F}}
\def\cG  {{\cal G}}
 
\def\cR{{\cal R}}

\def\be{{\beta}}
\def\de{{\delta}}
\def\la{{\lambda}}
\def\si{{\sigma}}

\def\BB{\mathbb{  B}}

\def\De{{\Delta}}

\def\Om{{\Omega}}
\def\al{{\alpha}}

\def\be{{\beta}}
\def\Ga{{\Gamma}}
\def\ga{{\gamma}}
\def\de{{\delta}}
\def\De{{\Delta}}

\def\si{{\sigma}}

\def\la{{\lambda}}

\def\vare{{\varepsilon}}
\def \eref#1{\hbox{(\ref{#1})}}

\def \eref#1{\hbox{(\ref{#1})}}

\def\si{{\sigma}}
\def\al{{\alpha}}

\def\cF  {{\cal F}}

\def\be{{\beta}}
\def\de{{\delta}}
\def\la{{\lambda}}
\def\si{{\sigma}}

\def\De{{\Delta}}

\def\Om{{\Omega}}
\def\al{{\alpha}}

\def\be{{\beta}}
\def\Ga{{\Gamma}}
\def\ga{{\gamma}}
\def\de{{\delta}}
\def\De{{\Delta}}

\def\si{{\sigma}}

\def\la{{\lambda}}

\def\e{{\varepsilon}}
\def \eref#1{\hbox{(\ref{#1})}}

\def\Om{{\Omega}}
\def\om{{\omega}}

\def\cP{{\cal  P}}
\def\bgamma{{\rm I}\! \Gamma  }

\begin{document}

\title{   Maximum Principle for General Controlled Systems Driven by
Fractional Brownian  Motions}
\author{{\sc Yuecai  Han}\thanks{
Y.  Han is partially supported   by NSFC grant 11071101, NSF of Jilin
grant 20101594 and the Science Research Fund grant 200903281 of
Jilin University}\,, \ \   {\sc Yaozhong Hu}\thanks{Y.  Hu is
partially supported by a grant from the Simons Foundation
\#209206.
\newline
Keywords: Stochastic optimal control,
backward differential
equations, maximum principle, fractional Brownian Motion,   controlled stochastic
differential systems driven fractional Brownian Motion,
Malliavan
calculus, partial information stochastic control.} \ \ and \  \  {\sc Jian Song}
\\
}
\date{}
\maketitle

\begin{abstract} We obtain a  maximum principle for stochastic
control problem of general controlled stochastic differential
systems driven by fractional Brownian motions (of Hurst parameter
$H>1/2$).  This maximum principle  specifies a system of equations
that the optimal control must satisfy (necessary condition for the
optimal control). This system of equations consists of a backward
stochastic differential equation driven by both fractional Brownian
motion and the corresponding underlying standard Brownian motion.   
In addition to this backward equation,
the maximum principle  also involves the Malliavin derivatives.  Our
approach is to use conditioning and Malliavin calculus.   To arrive
at our maximum principle we need to  develop some new results of
stochastic analysis of the controlled systems driven by fractional
Brownian motions via fractional calculus.  Our  approach of
conditioning and Malliavin calculus is also applied to classical
system driven by standard Brownian motion while the controller has only
partial information. As a straightforward consequence, the classical
maximum principle is also deduced in this more natural and simpler way.
\end{abstract}

\section{Introduction}
Fix a finite time horizon $T\in (0, \infty)$. Let $(\Om, \cF  , P)$ be
a basic probability space equipped with a right continuous
filtration $(\cF  _t)_{0\le t\le T}$ satisfying the usual conditions
(\cite{DM}).  Let $B^H=(B_1^H(t)\,, \cdots\,, B_m^H(t)\,, \ 0\le t \le T)$ be an
$m$-dimensional fractional Brownian motion of Hurst parameter $H\in
[\frac12, 1)$  (It is straightforward except notational complexity  to
allow $H$ to be different for different fractional Brownian
motions). This means that $B_j^H(t)$, $j=1, 2, \cdots, m$,  are
independent, continuous,  mean $0$ Gaussian processes with the
following covariance
\begin{equation}
\EE \left(B_i^H(t)B_j^H(s)\right)=\frac12 \de_{ij}\left(t^{2H}+s^{2H}-|t-s|^{2H}\right)\,,
\end{equation}
where  $\displaystyle \de_{ij}=\begin{cases} 1&\quad  \hbox{if}\quad i=j\\ 0& \quad
\hbox{if}\quad i\not=j\end{cases}$ is  the Kronecker symbol.

This process has been applied in many fields such as
hydrology, climatology, economics,  internet traffic analysis,
finance, and many other fields.  The stochastic analysis associated
with fractional Brownian motions has been extensively studied
recently. The stochastic differential equations driven by fractional
Brownian motions  have also been  considered by many researchers through
several approaches, see for example, through general rough path
analysis \cite{CQ}, \cite{FV}, \cite{hustochastics}, \cite{LQ} or through fractional
calculus   \cite{HN1}, \cite{HN2}, \cite{NR}. In particular, we refer   to the references therein.

Since  stochastic control is a main tool of applications of  stochastic
analysis  it is natural to consider the problem of stochastic control of systems driven by
fractional Brownian motions.  Along this direction there have been already some work.
In  \cite{Hu2}, \cite{HOS}    (see also \cite{BHOZ})
some specific stochastic control problems
  relevant to  mathematical finance have  been investigated.
  There is also a general sufficient condition of optimal control for general
  control problems in \cite{HO}.  The explicit optimal linear Markov control was obtained in
  \cite{HZ} by using the technique of completing squares and by using
  the Riccati equations.

However,  the problem of optimal control for general stochastic
systems driven by fractional Brownian motions is far away from being
considered as resolved. In fact,  there has been a lack of necessary
conditions   in a more general setting.  In this paper, we shall
fill this gap. More precisely,  we shall obtain a set of necessary
conditions that the optimal control must satisfy.

The theory of stochastic control of systems driven by standard
Brownian motions is very rich and has found many applications.
There have been mainly two general approaches toward the solutions.
One is the Hamilton-Jacobi-Bellman dynamic programming which
results in a highly nonlinear Hamilton-Jacobi-Bellman equation. The
study of viscosity solutions of this type of equations has
experienced an explosive growth in recent years. See \cite{FS},
\cite{YZ},   and the references therein for the stochastic control
related development. The study of viscosity solutions of the
nonlinear Hamilton-Jacobi-Bellman equation has become one  main
stream of partial differential equations. see \cite{CS}, \cite{L},
and many more other references for general discussion. Another
approach in stochastic control is the Pontryagin's maximum
principle. Starting with \cite{Bis1}, \cite{Bis2}, and \cite{Bis3},
backward stochastic differential equations (abbreviated as BSDEs) has
been used to describe the necessary (and sufficient) conditions that
the optimal control must satisfy. We also refer to \cite{FS},
\cite{P},   \cite{YZ}  and the references therein  for some other work.

The approach of  Bellman dynamic programming heavily depends on the
semigroup property of the underlying system (namely, the Markov
property of the underlying controlled stochastic processes).   In fact, one can also obtain the
Hamilton-Jacobi-Bellman equation for more general Markov processes,
see for example, \cite{nisio}. However, the fractional Brownian
motions are not Markov  except in the case of Brownian motion
($H=1/2$).  Thus,  it is natural to concentrate on extending the
Pontryagin's maximum principle to controlled system driven by fractional
Brownian motions.  The first main task is to find the appropriate
backward stochastic differential equations (an extension   of the
Riccati equation).  In this work, we obtain the backward stochastic
differential equations in a natural way, through the idea of conditioning.
This type of backward
stochastic differential equations involved terms  driven    both by
fractional Brownian motion and  by the underlying standard
Brownian motion.  In the classical standard Brownian motion case,
  researchers usually obtain this backward
stochastic differential equation by the    duality approach for
which one has to know the form of the BSDE in advance. Our approach
is motivated by a recent work \cite{HNS}  on linear BSDE driven by
standard Brownian motion by using Malliavin calculus.  We are
excited about this approach since it is very natural:   the BSDE is
deduced naturally without prior knowledge of the form of the BSDE!

Our approach is  also new in the classical setting of the controlled systems driven
by standard Brownian motions. The advantage of  our approach in the
classical setting of standard Brownian motion is that it also works for
stochastic control with {\it partial information}. Thus,  we also present
our approach to deduce the maximum principle in classical case,
first with partial information and then to give an alternative way
to deduce the classical maximum principle with complete information.  This is done in
Section 3.

To deduce the maximum principle for the controlled stochastic
system driven by fractional
Brownian motions, we need more results on stochastic analysis
of the controlled systems,  which has not been studied yet.  In
particular, we need to have the uniform H\"older continuity of the
solutions, and the differentiability of the solution with respect to
the control.  These results are of interest themselves. We shall
present these new results in Section 4. In  Section 5.1, using the idea in 
Section 3, we obtain a
necessary condition that the optimal control must satisfy when the
controller has only partial information. The stochastic control problem
with  partial information   is very important in finance, while not much
theory has been developed yet.  However, we refer to the work
\cite{HO} (for partial information linear quadratic control) and the
references therein. Section 5.2 aims to simplify the condition
obtained in Section 5.1 when the controller has complete information
available. The condition leads naturally to a new type of backward
stochastic differential equation driven by the underlying standard
Brownian motion and by the fractional Brownian motion. Besides the
complexity of the BSDE  which involves fractional Brownian motion and  the underlying standard
Brownian motion, the system of equations of maximum principle also
involves the Malliavin derivatives.  This system of equations of
maximum principle is very complex. However, this is expected since
the problem is much more complicated now. In the case of controlled
system driven by standard Brownian motion, the theory obviously
reduces to the classical maximum principle as demonstrated in
Section 3.

To obtain our maximum principle  we need some additional
results on fractional calculus and Malliavin calculus.  In Section 2, we introduce some notations and obtain some new results that we shall use.
  We also recall some necessary
notations from \cite{BHOZ}, \cite{DHP}, \cite{Hu1}, \cite{hustochastics},  
\cite{HN1}, and  \cite{HN2}     
 to
establish some new  results for controlled system driven by
fractional Brownian motions.   In
particular we shall establish some new important identity which is necessary
in our approach.

It is interesting to have some  examples  that can be solved by
using our new maximum principle.  However, as it is well-known,  the
explicit solutions of stochastic control are always difficult to
obtain even in the classical Brownian motion case.   See however,
\cite{BSW} for some general discussion. It is expected  that the
problem will be more complex.  In addition, this paper is already
 long. So,  we shall discuss some particular control problems
with this approach in the future.

\setcounter{equation}{0}
\section{Fractional calculus and Malliavin calculus}
\subsection{Fractional calculus}
In this section we recall some results from  fractional calculus.
Let $a,b\in \mathbb{R}$ with $a<b$ and let $ \alpha >0.$   The
left-sided (and right-sided)   fractional Riemann-Liouville
integrals of integrable  function $f$   is  defined  by
\[
I_{a+}^{\alpha }f\left( t\right) =\frac{1}{\Gamma \left( \alpha \right) }%
\int_{a}^{t}\left( t-s\right) ^{\alpha -1}f\left( s\right)
ds\,,\quad a\le t\le b \,,
\]%
[and
\[
I_{b-}^{\alpha }f\left( t\right) =\frac{\left( -1\right) ^{-\alpha
}}{\Gamma \left( \alpha \right) }\int_{t}^{b}\left( s-t\right)
^{\alpha -1}f\left( s\right) ds,
\]%
respectively], where $\left( -1\right) ^{-\alpha }=e^{-i\pi \alpha }$ and $%
\Gamma \left( \alpha \right) =\int_{0}^{\infty }r^{\alpha
-1}e^{-r}dr$ is the Euler gamma function.  The Weyl derivatives are
defined as
\begin{equation}
D_{a+}^{\alpha }f\left( t\right) =\frac{1}{\Gamma \left( 1-\alpha \right) }%
\left( \frac{f\left( t\right) }{\left( t-a\right) ^{\alpha }}+\alpha
\int_{a}^{t}\frac{f\left( t\right) -f\left( s\right) }{\left(
t-s\right) ^{\alpha +1}}ds\right)  \label{e.2.1}
\end{equation}%
and
\begin{equation}
D_{b-}^{\alpha }f\left( t\right) =\frac{\left( -1\right) ^{\alpha
}}{\Gamma \left( 1-\alpha \right) }\left( \frac{f\left( t\right)
}{\left( b-t\right)
^{\alpha }}+\alpha \int_{t}^{b}\frac{f\left( t\right) -f\left( s\right) }{%
\left( s-t\right) ^{\alpha +1}}ds\right)  \label{e.2.2}
\end{equation}%
where $a\leq t\leq b$.

For any $\be  \in (0,1)$, we denote by $C^{\be  }(a,b)$ the space of $%
\beta $-H\"{o}lder continuous functions on the interval $[a,b]$. We
will
make use of the notation%
\[
\left\| x\right\| _{a,b,\beta }=\sup_{a\leq \theta <r\leq b}\frac{%
|x_{r}-x_{\theta }|}{|r-\theta |^{\beta }},
\]%
and%
\[
\Vert x||_{a,b,\infty }=\sup_{a\leq r\leq b}|x_{r}|\,.
\]%
When $a$ and $b$ are clear, then we shall use
$\|x\|_\be=\left\| x\right\| _{a,b,\beta }$ and
$\|x\|_\infty=\left\| x\right\| _{a,b,\infty }$.

It is clear that when $f\in C^\be  (a, b)$ with $\be >\al $,  then
both   $D_{a+}^{\alpha }f\left( t\right) $ and $D_{b-}^{\alpha
}f\left( t\right)$   exist  and we have
\begin{equation}
\begin{cases}
\left|D_{a+}^{\alpha } f\left( t\right) \right| \le C\|f\|_\be
|t-a|^{  \be-\al}
&\qquad \hbox{if} \  \   f(a)=0 \\  \\
\left|D_{b-}^{\alpha }f\left( t\right) \right| \le C\|f\|_\be
|b-t|^{  \be-\al}
&\qquad \hbox{if} \ \   f(b)=0  \,.  \\
\end{cases}\label{e.2.3}
\end{equation}

The following fractional integration by parts formula and its
consequence are  needed in the sequel (see \cite{Za} for a proof).
\begin{proposition}
\label{p1} Suppose that $f\in C^{\lambda }(a,b)$ and $g\in C^{\mu
}(a,b)$ with $\lambda +\mu >1$. Let ${\lambda }>\alpha $ and $\mu
>1-\alpha $. Then the Riemann Stieltjes integral $\int_{a}^{b}fdg$
exists and it can be
expressed as%
\begin{equation}
\int_{a}^{b}fdg=(-1)^{\alpha }\int_{a}^{b}D_{a+}^{\alpha }f\left(
t\right) D_{b-}^{1-\alpha }g_{b-}\left( t\right) dt\,,
\label{e.2.4}
\end{equation}%
where $g_{b-}\left( t\right) =g\left( t\right) -g\left( b\right) $.
\end{proposition}

 From the definition \eref{e.2.1} of  $D_{a+}^{\alpha }f\left(
t\right)$ and \eref{e.2.3}, we have immediately
\begin{proposition}
\label{p2} Suppose that $f\in C^{\lambda }(a,b)$ and $g\in C^{\mu
}(a,b)$ with $\lambda +\mu >1$. Let ${\lambda }>\alpha $ and $\mu
>1-\alpha $. Then the Riemann Stieltjes integral $\int_{a}^{b}fdg$
exists and
\begin{eqnarray}
\left| \int_{a}^{b}fdg\right|
&\le & C\|g\|_\mu   \int_a^b \frac{|f(r)|}{(r-a)^{\al} } (b-r)^{\al+\mu-1} dr\nonumber\\
&&\quad +C \|g\|_\mu  \int_a^b \int_a^r
\frac{|f(r)-f(\tau)|}{|r-\tau|^{\al+1} }
 (b-r)^{\al+\mu-1}  d\tau dr \,.   \label{e.2.5}
\end{eqnarray}
\end{proposition}
\subsection{Stochastic  calculus for fractional Brownian motions}
Let $(W(t) =(W_1(t), \cdots, W_m(t)), 0\le t\le T)$ be an
$m$-dimensional standard Brownian motion. Let
\begin{equation}
Z_H(t,s)=\kappa_H \left[\left(\frac{t}{s}\right)^{H-\frac12
}(t-s)^{H-\frac12 }-(H-\frac12 )s^{\frac12 -H}\int_s^t u^{H-\frac32
}(u-s)^{H-\frac12 }du\right] \label{e.2.4a}
\end{equation}
with
\begin{equation}\label{kappah}
\kappa_H=\sqrt{\frac{2H\Ga(\frac32-H)}{\Ga(H+\frac12 )\Ga(2-2H)}}\,.
\end{equation}
 and define
\begin{equation}
B_j^H(t)=\int_0^t Z_H(t,s)dW_j(s)\,, \ 0\le t<\infty \,.
\label{e.2.3a}
\end{equation}
Then  from (2.2)-(2.4) of \cite{Hu1} we see that
$B^H(t)=(B_1^H(t)\,, \cdots, B^H_m (t)), 0\le t\le T) $ is an $m$
dimensional  fractional Brownian motion.   This means that
$B_j^H(t)\,, j=1, \cdots, m$  are independent Gaussian processes
with mean $0$ and variance given by
\[
\EE \left(B_j^H(t)B_i^H(s)\right)=\frac12 \de_{ij}
\left(r^{2H}+s^{2H}-|t-s|^{2H}\right)\,. 
\]
The stochastic integral
with respect to $B_j^H$ can be defined in a similar way as in \cite{Hu1}.  We shall use the
results in \cite{Hu1}.

Recall the operators   $\bgamma_{H, T}^* $  and
$\BB_{H,T}^*$  (see the equations (5.21) and (5.35) of \cite{Hu1})
\begin{equation}
\bgamma_{H,T} ^*   f(t):=(H-\frac12) \kappa_H t^{\frac12-H} \int_t^T
u^{H-\frac12} (u-t) ^{H-\frac32} f(u)du \,, \quad 0\le t\le T\,.
\label{e.2.8}
\end{equation}
and
\begin{equation}
\BB_{H,T}^* f(t)=-\frac{2H\kappa_1}{\kappa_H}
t^{\frac12-H}\frac{d}{dt} \int_t^T
(u-t)^{\frac12-H}u^{H-\frac12}f(u)du\,,\label{e.2.9}
\end{equation}
where $\kappa_H$ is defined in (\ref{kappah}) and $$\kappa_1=\frac{1}{2H\Gamma(H-\frac12)\Gamma(\frac32-H)}.$$

Let $\xi_1$, $\cdots$, $\xi_k$, $\cdots$ be an ONB of $L^2([0,
T])$ such that $\xi_k$, $k=1, 2, \cdots$ are smooth functions on
$[0, T]$.    We denote
$
\tilde \xi_{j,l}=\int_0^T \xi_j(t) dW_l(t) \,,
$\ where $j=1, 2, \cdots$, and $l=1, \cdots, m$.  \
Let $\cP $
be the set of all polynomials of  the standard Brownian motion $W$
over interval $[0, T]$.
Namely, $\cP$  contains
all elements of the form
\begin{equation}
F(\om)=f\left(\tilde \xi_{j_1, l_1} \,, \cdots \,, \tilde \xi_{j_n, l_n}
\right)\,,\label{e.2.10}
\end{equation}
where $f$ is a polynomial of $n$ variables.  
If $F$ is of the above form \eref{e.2.10},  then its Malliavin derivative
$D_s F$ is   defined as
\begin{equation}
D_s^l F=\sum_{i=1}^n \frac{\partial f}{\partial x_i}
\left(\tilde \xi_{j_1, l_1} \,, \cdots \,, \tilde \xi_{j_n, l_n}
\right)
 \xi_{j_i } (s) I_{ \{ l_i=l\}} \,,\quad 0\le s\le T\,.
\end{equation}
For any $F\in \cP$, we denote the following norm
\[
\|F\|_{1,  p}:=\|F\|_p+\sum_{l=1}^m  \left[\EE \left(\int_{[0 T] } |
D_t^l  F|^2dt \right)^{p/2}\right]^{1/p}
\,.
\]
It is easy to check  that $\|\cdot\|_{ 1,p}$ is a norm $p\in (1, \infty)$.
Let $\DD_{1,  p}$ denote the
Banach space obtained by completing $\cP$ under the norm
$\|\cdot\|_{1,p}$.  We can certainly define the higher derivatives.
But we only need the first derivatives in this paper.

For the above $\xi_j$'s  we can define $\eta_j=\BB_{H,T}^*  \xi_j$
and we denote by $\cP^H$ the set of all polynomial functionals of
$\tilde \eta_{j, l}:=\int_0^T \eta_j(t) dB_l^H(t)$.  For an  element
$G$ in $\cP^H$ of the following form
\begin{equation}
  G(\om)=g\left( \tilde \eta_{j_1, l_1} \,, \cdots \,, \tilde \eta_{j_n, l_n}
\right)\,,
\end{equation}
where $g$   is a polynomial of $n$ variables
we define its  Malliavin derivative
  $D_s^{H, l}  G$  by
\begin{equation}
D_s^{H, l}   G=\sum_{i=1}^n \frac{\partial g}{\partial x_i}
\left(\tilde \eta_{j_1, l_1} \,, \cdots \,, \tilde \eta_{j_n, l_n}
\right)
 \eta_{j_i } (s) I_{ \{ l_i=l\}}  \,,\quad 0\le s\le T\,.
\end{equation}
Similarly,  we can   define $\|\cdot\|_{H, 1,p}$ and  $\DD_{H, 1, p}$.

For fractional  Brownian motions, another different Malliavin derivative
is also very useful.
\begin{equation}
\DD_s^{ l} G=\int_0^T  \phi(s-r) D_r^{H, l} G dr \,,
\end{equation}
where
\begin{equation}
\phi(s)=H(2H-1) |s|^{2H-2}\,,\quad  0\le s\le T\,.
\end{equation}
It is well-known from (\cite{Hu1}, Theorem 6.23) that
\begin{prop}\label{p.2.3}
Let $\psi:[0,T]\otimes (\Om, \cF  , P^H)\rightarrow \RR$ be jointly
measurable and let $F\in D_{H, 1, 2}$. Then
\begin{equation}
\EE \left\{ F\int_a^b \psi_t d  B_j^H(t)\right\} =\int_a^b  \EE \left( \left[
\DD_t^j F\right]  \psi_t \right) dt\  \label{ito-expectation}
\end{equation}
\end{prop}
The following proposition from  (\cite{DHP}, Theorem 3.9) will also
be used in this the sequel.
\begin{prop}\label{p.2.4}  Let $\psi$ be a jointly measurable stochastic process
which has Malliavin derivative and  let the following hold:
\[
\EE \left[ \int_a^b\int_a^b |\psi_s\psi_t|\phi(s-t) dsdt
+\int_a^b\int_a^b |\DD_s \psi_t|^2 dsdt\right] <\infty\,.
\]
Then
\begin{equation}
 \int_a^b \psi(t) d^\circ B_j^H(t)= \int_a^b \psi(t) d B_j^H(t)
 +   \int_a^b \DD_t^{j} \psi(t) dt\,.
 \  \label{ito-path}
\end{equation}
\end{prop}
The following proposition is easy consequence of the above identity (\ref{ito-expectation})
and the relation between path-wise integral and the stochastic integral obtained
by using the Wick product (\ref{ito-path}).
\begin{prop} If $F$ satisfies the condition in Proposition
\ref{p.2.3} and $\psi$ satisfies the conditions in Propositions
\ref{p.2.3} and \ref{p.2.4},  then
\begin{equation}
\EE \left\{ F\int_a^b \psi_t d^\circ B_j^H(t)\right\} =\int_a^b\EE \left(
\DD_t^j \left[F \psi_t\right] \right) dt\,. \label{expectation}
\end{equation}
\end{prop}

If $F$ is a nice functional of the fractional Brownian motions
$B^H=( B_1^H, \cdots, B_m^H)$,   then it is also a functional of
standard Brownian motions $W=(W_1, \cdots, W_m)$. So for this $F$ we
can compute both the Malliavin derivatives $D_s^l$ and $D_s^{H,l}$.
The relation between these two Malliavin derivatives are useful
later in this paper.   Next, we shall find this relation.

If $f$ is a nice function on $[0, T]$,  then
$\int_0^T f(t)dW_j(t)$ and $\int_0^T f(t)dB_j^H(t)$ are well-defined
and from page 45 of \cite{Hu1},  we have
\[
\int_0^T f(t) d B_j^H(t)=\int_0^T (\bgamma_{H, T}^* f)(t) dW_j(t)
\]
and
\[
\int_0^T f(t) dW_j(t)=\int_0^T (\BB_{H, T}^* f)(t) dB^H_j(t)\,,
\]
where $\bgamma_{H, T}^*$ and $\BB_{H, T}^*$ are defined by \eref{e.2.8}
and \eref{e.2.9}.
If $F$ is given   by \eref{e.2.10},   then $F$ can also be represented by
\[
F=f\left(\int_0^T (\BB_{H, T}^* \xi_{j_1})(t) dB_{l_1}^H(t)\,,
\cdots\,,  \int_0^T (\BB_{H, T}^* \xi_{j_n}) (t) dB_{l_1}^H(t)
\right)\,.
\]
Thus as a functional of $B^H$,  its Malliavin derivative
$D_s F$ is   defined as
\begin{eqnarray*}
D_s^{H,l}  F
&=&\sum_{i=1}^n \frac{\partial f}{\partial x_i}
\left(\int_0^T (\BB_{H, T}^* \xi_{j_1})(t) dB_{l_1}^H(t)\,,
\cdots\,,  \int_0^T (\BB_{H, T}^* \xi_{j_n}) (t) dB_{l_1}^H(t)
\right)\\
&&\qquad
 (\BB_{H, T}^* \xi_{j_i}) (s) I_{ \{ l_i=l\}} \,,\quad 0\le s\le T\,.
\end{eqnarray*}
Thus for $F\in\cP$,  we have
\begin{eqnarray*}
D_s^{H,l}  F
&=& \BB_{H, T}^* D_\cdot^l  F(s) \\
&=&-\frac{2H\kappa_1}{\kappa_H}
s^{\frac12-H}\frac{d}{ds} \int_s^T
(u-s)^{\frac12-H}u^{H-\frac12} D_u^lFdu
\end{eqnarray*}
and
\begin{eqnarray*}
\DD^l_t  F
 &=&-\frac{2H\kappa_1}{\kappa_H} H(2H-1) \int_0^T
s^{\frac12-H} |t-s|^{2H-2}\left( \frac{d}{ds} \int_s^T
(u-s)^{\frac12-H}u^{H-\frac12} D_u^lFdu\right) ds\,.
\end{eqnarray*}
By a limiting argument, we have the following
\begin{prop} \label{p.2.6}   If $F\in {  \DD_{1, p}\cap \DD_{H, 1, p}}$,  then
\begin{eqnarray}
D_s^{H,l}  F &=&-\frac{2H\kappa_1}{\kappa_H}
s^{\frac12-H}\frac{d}{ds} \int_s^T
(u-s)^{\frac12-H}u^{H-\frac12} D_u^lFdu
\label{e.2.18}\\
\DD^l_t  F
 &=&c_{1, H}  \int_0^T
s^{\frac12-H} |t-s|^{2H-2}\left( \frac{d}{ds} \int_s^T
(u-s)^{\frac12-H}u^{H-\frac12} D_u^lFdu\right) ds\,,\nonumber\\
\label{e.2.19}
\end{eqnarray}
where $c_{1, H}=-\frac{2H^2(2H-1) \kappa_1}{\kappa_H}  $.
\end{prop}

\begin{prop} \label{p.2.5}  Let  $F_t=\sum_{j=1}^m \int_0^t f_j(s)  dB_j^H(s)$ and $G_t= \sum_{j=1}^m \int_0^t g_j(s)  dW_j(s) $, where $f_1, \cdots, f_m$ satisfy
the conditions in Proposition \ref{p.2.4} and let $g_1, \cdots, g_m$ be  continuous
adapted processes.    Then
\begin{equation}
d(F_tG_t)=FdG_t+G_tdF_t\,.
\end{equation}
\end{prop}
\begin{proof} We can prove this proposition in the same way as the proof for
Theorem 2.1 of \cite{HZ}.
\end{proof}

\section{Systems driven by Brownian motions}
\setcounter{equation}{0}


In this section, we shall use our
approach of conditioning and Malliavin calculus to deduce the maximum principle for partially observed controlled system driven by standard Brownian
motions.  The complete information case will also be deduced.
In this classical case, our approach seems to be new,
straightforward and simpler.

\subsection{General stochastic control with partial information} In this subsection
we shall obtain a maximum principle when only partial information is available.
This means that the control $u_t$ may not necessarily depends on full information
determined by the Brownian motions  $\{W(t): 0\le t\le T\}$.    We assume in this subsection
that $\left(\cG_t\right)_{0\le t\le T} $ is any right continuous filtration which is contained
in the $\si$-algebra $\left(\cF  _t\right)_{0\le t\le T} $ generated by  $\{W(t): 0\le t\le T\}$.

The space of admissible controls is defined as
\begin{eqnarray*}
U[0,T]
&\stackrel{\Delta}{=}& \Big\{u: [0,T]\times \Omega\to \RR^d\ |
\ ~ u~ {\rm
is}~ {\cal G}\hbox{-}{\rm adapted \ stochastic \ process} \quad\\
&&\quad\qquad  {\rm and}\quad \EE \left( \int_0^T|u(t)|^2dt\right) <+\infty
\Big\}.
\end{eqnarray*}

To describe our stochastic control problem,  we need to introduce four functions
$b$, $\si$, $l$,   and $h$. They are assumed to satisfy the following conditions.
\begin{description}
\item[(H1)]   The functions
$b :[0,T]\times \RR^n\times \RR^d \to \RR^n$,    $\sigma=(\sigma^1,\cdots,\sigma^m):
[0,T]\times \RR^n\times \RR^d \to \RR^{n\times m}$,  and  $h: \RR^n\to \RR$ are continuous with respect variables $t$, $x$, and $u$
and are
continuously differentiable with respect to $x, u$ for all $t\in [0, T]$.
%
%
%

 Denote
\[
b_x(t, x, u)=\left( \frac{\partial b_i(t,x,u)}{\partial x_j}\right)_{1\le i, j \le n}\,,
\quad b_u(t, x, u)=\left( \frac{\partial b_i(t,x,u)}{\partial u_j}\right)_{1\le i\le n, 1\le  j \le d}
\]

\item[(H2)]   We assume that
there is a constant $C>0$ such that
\[
 \sup_{t\in [0, T], x\in\RR^n\,, u\in \RR^d} |b_x(t, x, u) |+ | b_u(t, x, u)|
 +\sum_{j=1}^m |\sigma_x^j (t, x, u)|+\sum_{j=1}^m |\sigma_u^j (t, x, u)| \le C\,.
\]

\item[(H3)] We assume that $ l:[0,T]\times \RR^n\times \RR^d \to \RR $
and $h:\RR^n\rightarrow \RR$ be continuously differentiable with bounded derivatives.

[In what follows throughout this paper, we shall use $C$ to denote a generic constant which may have different value at different
occurrences.]
\end{description}

The  controlled stochastic control system is described as the following
stochastic differential equation:
\begin{eqnarray}
 \left\{\begin{array}{rcl}
dx(t)&=&b(t,x(t),u(t))dt+\sum\limits_{j=1}^m\sigma^j(t,x(t),u(t))dW_j(t), \\
x(0)&=&x_0\,, \end{array}\right. \label{stateeq}
\end{eqnarray}
where $x_0$ is a given vector in $\RR^n$.
For a given  $u\in U[0, T]$
the existence and uniqueness of the solution $x^u(t)$  to the above equation
follows from standard theory of stochastic differential equations. For simplicity,
we sometimes omit the explicit dependence of $x(t)$ on $u$,  namely, we write
$x(t)=x^u(t)$.
The cost functional  we shall deal with is
\begin{eqnarray}\label{e.3.2}
J(u(\cdot))=J(x^u(\cdot), u(\cdot))=\EE
\left\{\int_0^Tl(t,x^u(t),u(t))dt+h(x^u(T))\right\}.
\end{eqnarray}

The first optimal control problem  studied in this paper is   to
minimize the cost functional $J(u(\cdot))$ over $U[0,T]$.   This
means that we want to find the  optimal control $u^*(\cdot) \in U[0,
T]$ satisfying
\begin{equation}
J(u^*(\cdot)) =\inf\limits_{u(\cdot)\in U[0,T]}J(u(\cdot))\,. \label{e.3.3}
\end{equation}
If such an optimal control $u^*(\cdot)$ exists,  then
the corresponding state  $x^*(\cdot)=x^{u^*} (\cdot)$ is called  an  optimal state process.
  $(x^*(\cdot), u^*(\cdot))$ is   called an optimal pair   for  the optimal control problem described by \eref{stateeq},
  \eref{e.3.2}    and \eref{e.3.3}.

Assume   $(x^*(\cdot),u^*(\cdot))$  is  an  optimal  pair.  Namely, $J(u(\cdot))$ attains
its minimum  at $u^*(\cdot)$.    From the general theory of functional analysis, we know that $u^*(\cdot)$ is a critical point of $J(u(\cdot))$.  We are interested in finding a necessary condition  satisfied by
$(x^*(\cdot),u^*(\cdot))$. For any $\e\in \RR$ and $ v(\cdot)\in U[0,T]$,  let $u^\e(\cdot)=u^*+\e v(\cdot)$.  It is easy to see that  $u^\e(\cdot) \in  U[0,T]$, and there is a unique $x^\e(\cdot)$ satisfying the
state equations (\ref{stateeq})  with control $u$ replaced by $u^\e(\cdot)$.

Denote
\begin{eqnarray}
b_x^*(t)
&=& b_x(t, x^*(t), u^*(t))\,,\quad b_u^*(t)=b_u(t, x^*(t), u^*(t))\,,\nonumber\\
\sigma_x^{j, *}(t)&=&  \sigma_x^{j,*}(t, x^*(t), u^*(t))\,,
\quad \sigma_u ^{j, *}(t)= \sigma_u^{j, *}(t, x^*(t), u^*(t))\,.
\end{eqnarray}
The following lemma can be   proved easily and it is necessary in
obtaining the maximum principle for the stochastic control problem
\eref{stateeq}-\eref{e.3.3}.
\begin{lemma}\label{lemma3}
The limit  $y(t)=\lim\limits_{\varepsilon\to 0}\frac{x^\e(t)-x^*(t)}{\varepsilon}$ exists in $L^2$  and
$y(t)$ satisfies the following equations:
\begin{eqnarray}\label{variationeq1}
\left\{\begin{array}{rcl}
dy(t)&=&[b_x^*(t)y(t)+b_u^*(t)v(t)]dt+
\sum\limits_{j=1}^m[\sigma_x^{j, *}(t)y(t)+\sigma_u^{j, *}(t)v(t)]dW_j(t),\\[-0.5em]\\
y(0)&=&0.\end{array}\right.
\end{eqnarray}
\end{lemma}
We need
to solve the above linear stochastic differential equations (with random coefficients).
In particular,   we need to express $y$  in  an explicit form of $v$.  For this reason
we  consider   the following   linear matrix-valued
stochastic differential equations.
\begin{eqnarray}
\left\{\begin{array}{rcl}
d\Phi(t)&=&b_x^*(t)\Phi(t)dt+\sum\limits_{j=1}^m\sigma_x^{j, *}(t)\Phi(t)dW_j(t),\\
\Phi(0)&=&I,
\end{array}\right. \label{e.3.6}
\end{eqnarray}
From the basic theory of stochastic differential equations, it is well-known that
this equation has a unique solution, denoted by $\Phi(t)$.
It is easy  to verify that   $\Phi^{-1}(t) $  exists and is  the unique solution of the following stochastic differential equations  (\cite{YZ}):
\begin{eqnarray}\left\{\begin{array}{rcl}
d\Phi^{-1}(t)&=&\Phi^{-1} (t)\left(-b_x^*(t)+\sum\limits_{j=1}^m\left(\sigma_x^{j, *}(t)\right)^2\right)dt-\sum\limits_{j=1}^m\Phi^{-1}
(t)\sigma_x^{j, *}(t)dW_j(t),\\
\Phi^{-1}(0)&=&I\,.
\end{array}\right.\label{e.3.7}
\end{eqnarray}

From the  It\^o's formula  we can obtain the solution of  the equation {\rm  (\ref{variationeq1})}
by using
 $\Phi(t), \Phi^{-1}(t)$.
\begin{eqnarray}
y(t)&=&\Phi(t)\int_0^t\Phi^{-1}(s)\left(b_u^*(s)-\sigma_x^*(s)\sigma_u^*(s)\right)v(s)ds\nonumber\\
&&\qquad
+\sum\limits_{j=1}^m\Phi(t)\int_0^t\Phi^{-1}(s)\sigma_u^{j, *}(s)v(s)dW_j(s)\,,   \label{yeq1}
\end{eqnarray}
where
\begin{equation}
\sigma_x^*(s)\sigma_u^*(s)=\sum_{j=1}^m \sigma_x^{j,*}(s)\sigma_u^{j,*}(s)
\,.
\end{equation}

Since $u^*(\cdot)$ is an optimal control,  it is a critical point of the functional $J(u(\cdot))$, namely,
we have
\begin{eqnarray}
 \left.\frac d{d\e}J(u^*(\cdot)+\e v(\cdot))\right |_{\e=0}
 &=& 0\,.
\label{e.3.9}
\end{eqnarray}
But
\begin{eqnarray*}
 \left.\frac d{d\e}J(u^*(\cdot)+\e v(\cdot))\right |_{\e=0}
 &=& \EE \left(\int_0^T\left(l_x^{*\top}(t)y(t)+l_u^{*\top}(t)v(t)\right)dt\right)
+\EE \left( h_x^\top(x^*(T))y(T)\right)  \,.
\end{eqnarray*}
Using  (\ref{yeq1})  for the solution $y(t)$,   we have
\begin{eqnarray*}
 &&\left.\frac d{d\e}J(u^*(\cdot)+\e v(\cdot))\right |_{\e=0}\\
 &=& \EE \int_0^T\left(l_x^{*\top}(t)\Phi(t)\left[ \int_0^t\Phi^{-1}(s)
 \left(b_u^*(s)-\sigma_x^*(s)\sigma_u^*(s)\right)
v(s)ds\right] +l_u^{*\top}(t)v(t)\right)dt\\
&~&+\EE \int_0^T\left(\sum\limits_{j=1}^ml_x^{*\top}(t)\Phi(t)\left[ \int_0^t\Phi^{-1}(s)
\sigma_u^{j, *}(s)v(s)dW_j(s)\right] \right)dt\\
&~&+\EE \left\{h_x^\top(x^*(T))\Phi(T)\left[ \int_0^T\Phi^{-1}(s)(b_u^*(s)-\sigma_x^*(s)
\sigma_u^*(s))v(s)ds\right] \right\}\\
&~&+\sum\limits_{j=1}^m\EE \left\{h_x^\top(x^*(T))\Phi(T)\left[ \int_0^T\Phi^{-1}(s)
\sigma_u^{j, *}(s)v(s)dW_j(s)\right] \right\}\,.
\end{eqnarray*}
Now we make use of the following identity from Malliavin calculus:
$$
\EE \left(F\int_0^Tg(t)dW_j(t)\right)=\EE \int_0^T\left({ D}  _t^j F\right) g(t)dt\,.
$$
Combining this identity and Fubini type theorem,   we have
\begin{eqnarray}
&&\left.\frac d{d\e}J(u^*(\cdot)+\e v(\cdot))\right |_{\e=0}\nonumber \\
&=&\EE \int_0^T\left(\int_s^T l_x^{*\top}(t)\Phi(t)
\Phi^{-1}(s)(b_u^*(s)-\sigma_x^*(s)\sigma_u^*(s))v(s)dt\right)ds\nonumber \\
&~&+\int_0^T\EE \left(\sum\limits_{j=1}^m\int_0^t{ D}  _s^j
\left(l_x^{*\top}(t)\Phi(t)\right)\Phi^{-1}(s)\sigma_u^{j,*} (s)v(s)ds\right)dt\nonumber \\
&~&+\EE \int_0^Tl_u^{*\top}(s)v(s)ds+\EE
\int_0^T h_x^\top(x^*(T))\Phi(T)\Phi^{-1}(s)(b_u^*(s)-\sigma_x^*(s)\sigma_u^*(s))v(s)ds\nonumber \\
&~&+\sum\limits_{j=1}^m\EE \int_0^T{ D}  _s^j
\left(h_x^\top(x^*(T))\Phi(T)\right)\Phi^{-1}(s)\sigma_u^{j,*} (s)v(s)ds\,.
\label{e.3.10}
\end{eqnarray}
Denote
\begin{eqnarray}
\Psi(T,s) &:=&  \left(\int_s^Tl_x^{*\top}(t)\Phi(t)dt\right)\Phi^{-1}(s)(b_u^*(s)
-\sigma_x^*(s)\sigma_u^*(s))+l_u^{*\top}(s)
\nonumber  \\
&~& +\sum\limits_{j=1}^m\left(\int_s^T{ D}  _s^j \left(l_x^{*\top}(t)\Phi(t)dt\right)
\Phi^{-1}(s)\sigma_u^{j, *}(s)+{ D}  _s^j \left(h_x^\top(x^*(T))\Phi(T)\right)\Phi^{-1}(s)
\sigma_u^{j, *}(s)\right)\, \nonumber\\
&~&+h_x^\top(x^*(T))\Phi(T)\Phi^{-1}(s)(b_u^*(s)-\sigma_x^*(s)\sigma_u^*(s)).
\label{e.3.11}
\end{eqnarray}
Then from \eref{e.3.9} and \eref{e.3.10}   we have
\begin{eqnarray*}
\left.\frac d{d\e}J(u^*(\cdot)+\e v(\cdot))\right |_{\e=0}
&=&\EE \left[ \int_0^T \Psi(T, s) v(s)ds\right]\\
&=& \EE \left[ \int_0^T \EE\left\{ \Psi(T, s)\big| \cG_s\right\}  v(s)ds\right] =0\,.
\end{eqnarray*}
Since the above identity holds true for all $\left(\cG_t\right)_{
0\le t\le T} $ adapted process $v\in U[0, T]$, we have
\begin{equation}
 \EE\left\{ \Psi(T, s)\big| \cG_s\right\} = 0 \quad \forall \ 0\le s\le  T\,. \label{e.3.13}
 \end{equation}
 We can also write the above equation as
 \begin{equation}
 \EE\left\{ \Psi(T, s)^\top \big| \cG  _s\right\} = 0 \quad \forall \ 0\le s\le  T\,. \label{e.3.14}
 \end{equation}
Denote $\EE^{\cG  _t}(X)= \EE\left\{ X\big| \cG  _t\right\} $.  Then  we can state \eref{e.3.14} as
the following general maximum principle (e.g. the equation \eref{h_u00}  below).
\begin{theorem}\label{t1}
Let $(x^*(\cdot), u^*(\cdot))$ be an optimal pair of the
control problem \eref{stateeq}-\eref{e.3.3}.  Define
\begin{eqnarray*}
\left\{\begin{array}{rcl}
P(t)&=& {\Phi^\top} ^{-1}(t)
\int_t^T {\Phi^\top}(s)l_x^*(s)ds+ {\Phi^\top}^{-1}(t) {\Phi^\top}(T)h_x(x^*(T))\, , \\
Q_j(t)&=&-\sigma_x^{j,*}(t)P(t)+ {\Phi^\top}^{-1}(t)
 { D}  _t\left( {\Phi^\top}(T)h_x(x^*(T))\right)+\int_t^T{ D}  _t\left( {\Phi^\top}(s)l_x^*(s)\right)ds\,.
\end{array}\right.
\end{eqnarray*}
Then
\begin{eqnarray}\label{h_u00}
\EE^{\cG  _t}\left[b_u^{*\top}(t) P(t)+\sum\limits_{j=1}^m\sigma_u^{j,*\top}(t) Q_j(t)+l_u^*(t)
\right]=0\ \quad\forall  \ \ t\in [0, T]
\end{eqnarray}
almost surely.
\end{theorem}

\subsection{Stochastic control with complete information}

If $\cF_t=\si(W_1(s)\,,\cdots\,,
W_m(s)\,,  0\le s\le t)$ is the $\si$-algebra generated
by the Brownian motion $W(s)=(W_1(s)\,, \cdots\,, W_m(s))$,
then the above equation \eref{h_u00}  for the maximum principle  can be simplified.

First note that $b_u^{*\top}(t)$, $  \sigma_u^{1,*\top}(t)$,
$\cdots$, $  \sigma_u^{1,*\top}(t)$, and $l_u^*(t)$ are $\cF  _t$-adapted,
then the equation \eref{h_u00}  can be written as
\begin{equation}
b_u^{*\top}(t) p(t) +\sum\limits_{j=1}^m\sigma_u^{j,*\top}(t)
q_j(t)  +l_u^*(t)
 =0\ \quad\forall  \ \ t\in [0, T]\,,
\end{equation}
where we
denote $p(t)=\EE^{\cF  _t}\left[ P(t)\right]$ and $q_j(t)=
\EE^{\cF  _t}\left[Q_j(t)\right]$.  From the definition of $P(t)$ and $Q_j(t)$,   we have
\begin{eqnarray}\label{pqeq}
\left\{\begin{array}{rcl}
p(t)
&=&{\Phi^\top}^{-1}(t)\EE ^{{\cal F}_t}\left[\int_t^T\Phi^\top(s)l_x^{*}(s)ds
+\Phi^{\top}(T)h_x(x^*(T))\right],\\
q_j(t)&=&-\left(\sigma_x^{j, *}(t)\right) p(t)+{\Phi^\top}^{-1}(t)
\EE ^{{\cal F}_t}\Big[{ D}  _t\left(\Phi^\top(T)h_x(x^*(T))\right)
\\
&&\qquad\qquad \left.+\int_t^T{ D}  _t
\left(\Phi^\top(s)l_x^*(s)\right)ds\right],\qquad j=1,\cdots,m.
\end{array}\right.
\end{eqnarray}
\begin{lemma} \label{lemma2.3} If $p(t)$ and $q_j(t)$ are defined as above, then
\begin{equation}
q_j(t)=D_t^j p(t)\,.
\end{equation}
\end{lemma}
\begin{proof} From \eref{e.3.7}  we see that
\[
D_t^j\Phi^{-1}(t)=-\Phi^{-1}(t) \si_x^{j, *}(t)\,.
\]
On the other hand, from Proposition 1.2.8 of \cite{N}, we see that
\[
  \EE^{\cF  _t}(D_t^j X)=D_t^j \EE^{\cF  _t}(X)\,.
\]
This proves the lemma easily.
\end{proof}

From the equation  (2.11) of \cite{HNS}, we see that $p(t)$ and $(q_1(t)\,, \cdots\,, q_m(t))
$ is the unique solution
of the following backward stochastic differential equations.
\begin{eqnarray}\label{bsde1}
\left\{\begin{array}{rcl}
-dp(t)&=&(b_x^{*\top}(t)p(t)+\sum\limits_{j=1}^m\sigma_x^{j, *\top}(t)q_j(t) +
l_x^*(t))dt-\sum\limits_{j=1}^mq_j(t)dW_j(t)\,, \\
p(T)&=&h_x(x^*(T))\,.
\end{array}\right.
\end{eqnarray}

\begin{theorem}\label{t2}
Let $(x^*(\cdot), u^*(\cdot))$ be an optimal pair of the
control problem \eref{stateeq}-\eref{e.3.3}.  Let $p(t)$ and $(q_1(t), \cdots\,, q_m(t))$
be the unique solution pair to \eref{bsde1}.
Then
\begin{eqnarray}\label{h_u0}
\ b_u^{*\top}(t) p(t)+\sum\limits_{j=1}^m\sigma_u^{j,*\top}(t)  q_j(t) +l_u^*(t)
 =0\ \quad\forall  \ \ t\in [0, T]
\end{eqnarray}
almost surely.
\end{theorem}

\begin{remark}  The equations \eref{stateeq}, \eref{bsde1}, and
\eref{h_u0} is a system of coupled forward-backward stochastic
differential equations.  Usually they can be used to determine the
optimal control $u^*$ and the corresponding optimal state $x^*$. For
the convenience we can write them together as
\begin{equation}
\begin{cases}
dx^*(t)=b(t,x^*(t),u^*(t))dt+\sum\limits_{j=1}^m\sigma^j(t,x^*(t),u^*(t))dW_j(t), \\
x^*(0)=x_0\,,\\
-dp(t)=(b_x^{*\top}(t)p(t)+\sum\limits_{j=1}^m\sigma_x^{j, *\top}(t)q_j(t) +
l_x^*(t))dt-\sum\limits_{j=1}^mq_j(t)dW_j(t)\,, \\
p(T)=h_x(x^*(T))\\
b_u^{*\top}(t) p(t)+\sum\limits_{j=1}^m\sigma_u^{j,*\top}(t) q_j(t)
+l_u^*(t)
 =0\,,
 \end{cases}\label{bmmax}
\end{equation}
where $0\le t\le T$.
\end{remark}
The system of coupled forward-backward stochastic differential
equations  \eref{stateeq}, \eref{bsde1}, and \eref{h_u0}  can also be
written by using the so-called Hamiltonian.   Let
\begin{eqnarray*}
H(t,x,u,p,q)&:=&b(t,x,u)^\top
p(t)+\sum\limits_{j=1}^m\sigma^{j, \top}(t,x,u) q_j(t)+l(t,x,u),\\
&~& (t,x,u,p,q)\in [0,T]\times \RR^n\times \RR^d\times \RR^n\times
\RR^{n\times m}\,.
\end{eqnarray*}
Then the maximum principle \eref{bmmax}  can be
restated as
\begin{equation}
\begin{cases}\displaystyle
dx^*(t)=\frac{\partial }{\partial p} H(t,x^*(t),u^*(t),p(t),q(t))
dt+\sum_{j=1}^m \frac{\partial }{\partial q_j} H(t,x^*(t),u^*(t),p(t),q(t)) dW_j(t)  \\
x^*(0)=x_0\,,\\
\displaystyle  -dp(t)=\frac{\partial }{\partial x}
H(t,x^*(t),u^*(t),p(t),q(t)) dt
-\sum\limits_{j=1}^mq_j(t)dW_j(t)\,, \\
p(T)=h_x(x^*(T))\\ \\
\displaystyle   \frac{\partial }{\partial u} H
(t,x^*(t),u^*(t),p(t),q(t)) = 0 \,,
 \end{cases}
\end{equation}
where $0\le t\le T$.

\setcounter{equation}{0}

\section{Controlled stochastic differential systems driven by fractional  Brownian motions}

To obtain our main results of maximum principle for the system driven by fractional Brownian motions,  we need to develop some new results for controlled stochastic differential equations driven by fractional Brownian motions.  Let us recall that $B^H(t)=(B_1^H(t), \cdots, B_m^H(t)), 0\le t\le T$, is  an $m$-dimensional Brownian motion. Let $(\cG_t\,, 0\le t\le T)$ be a right continuous filtration contained in the filtration $(\cF_t\,, 0\le t\le T)$ generated by fractional  Brownian motions $B^H(t)$.

 First,   let us define our space of admissible controls:
\begin{eqnarray}
U[0,T]
&:=&\Bigg\{u| u: [0,T]\times \Omega \to \RR^d, \mathcal G\hbox{-adapted}, u\in C^\mu[0, T]\  \hbox{for some}\ \mu>1-H\,,
\nonumber\\
&&\hbox{there exist constant $C>0$,   $c>0$,  and $\be<H$}  \nonumber\\
&&\quad
\hbox{such that} \ \  \|u\|_{\mu, 0, T}\le C
e^{c \sum_{j=1}^m \|B_j^H\|_{\be, 0, T}}\nonumber\\
&&\quad  \hbox{and}\ \ \int_0^T\int_0^T \EE \left(|D_s^H u(t)|^2\right) dsdt<\infty
\Bigg\}\,.  \label{e.4.1}
\end{eqnarray}
%
Consider the following controlled stochastic differential equation driven by fractional
Brownian motion:
\begin{eqnarray}\left\{\begin{array}{rcl}
dx(t)&=&b(t,x(t),u(t))dt+\sum\limits_{j=1}^m\sigma^j(t,x(t),u(t))  d^\circ B_j^H(t),\\
x(0)&=&x_0\,. \end{array}\right.\label{5.3.2}
\end{eqnarray}
Here the integral with respect to fractional Brownian motion is in the pathwise sense (or the Stratonovich type integral).

We assume that  $b: [0,T]\times \RR^n\times \RR^d \to \RR^n$,
$\sigma:[0,T]\times \RR^n \times \RR^d   \to \RR^{n\times m}$
 are  some given continuous  functions
satisfying the following conditions.

\begin{description}
\item[(H4)] $b(t, x, u)$ is continuously differentiable with respect to $x$ and $u$.  Moreover,
there exists  a constant $L  $ such that the following properties hold:
\begin{eqnarray}
&&|b_x(t, x, u)|+|b_u(t, x, u)| \le L \,,
\nonumber\\
&& |b_x(t, x, u)-b_x(t, y, v)|+|b_u(t, x, u)-b_u(t, y, v)| \le L  (|x-y|+|u-v|) \,,  \nonumber\\
&&\qquad ~~ \forall x \ \in \RR^n, ~~ \forall \ \ u\in\RR^d, ~~ \forall t\in [0,T]\,.
\end{eqnarray}
\item[(H5)] $\sigma(t,x,u)$ is twice continuously differentiable in $x$ and $u$ and
there exist some constants $1-H<\ga<1$ and $0<\delta\le 1$,  and $L$  such that for each $i=1,\cdots, n$:
\begin{itemize}
\item[(i)]   The partial derivatives of $\sigma$ with respect to $x$ and $u$ are bounded:
\begin{eqnarray*}
  |\sigma_x(t, x, u)|+|\sigma_u(t, x, u)| &\le & L\,, \\
 |\sigma_{xx}(t, x, u)|+|\sigma_{uu}(t, x, u)|+\si_{xu}(t, x, u) | &\le &   L\,.
\end{eqnarray*}
\item[(ii)] H\"older continuity in time: $\forall x\in \RR^n, u\in \RR , \forall t,s\in [0,T]$,
\begin{align*}
 |\sigma(t,x,u)-\sigma(s,x,u)|+|\partial_{x }\sigma(t,x,u)-\partial_{x  }\sigma(s,x,u)|&&\\
 +|\partial_{u}\sigma(t,x,u)-\partial_{u}\sigma(s,x,u)|+ |\sigma_{xx}(t,x,u)-\sigma_{xx}(s,x,u)| &&\\
   +|\sigma_{xu}(t,x,u)-\sigma_{xu}(s,x,u)| + |\sigma_{uu}(t,x,u)-\sigma_{uu}(s,x,u)|
&\le & L |t-s|^\ga.
\end{align*}
\item[(iii)]   Lipschitz  continuity of second derivatives with respect to state and control variables.
\begin{eqnarray*}
&&|\sigma_{uu}(t,x,u)-\sigma_{uu}(t,y,v)|  +|\sigma_{xu}(t,x,u)-\sigma_{xu}(t,y,v)|\\
&\le&  L\left(|x-y|  +|u-v| \right) \,.
\end{eqnarray*}
\end{itemize}
\end{description}


If $b$ and $\si$ satisfy the above assumptions  (H4) and (H5) and if $u$ is an admissible control,
then   the coefficients $b(t, x)=b(t, u(t), x)$ and $\si(t,x)=\si(t, u(t), x)$ satisfy   the conditions $(H_1)$, $(H_2)$, and $(H_3) $  of \cite{NR}.     Thus it follows that    for any  admissible control $u$,
the controlled stochastic differential equation (\ref{5.3.2}) has a unique solution (see also  \cite{HN1}),
denoted by $x_t^u$.
   Moreover,  for any $1-H<\al<1/2$,    the solution is ${1-\alpha}$  H\"older continuous almost surely, namely,
   $|x(r)-x(\tau)|\le c_0|r-\tau|^{1-\alpha}$ almost surely
(where $c_0$ may depends on $B_\cdot^H$).
The solution $x_t^u$ of above equation \eref{5.3.2}  depends on $u$.  But
to simplify the notation we often omit its explicit dependence on $u$ and write
$x_t=x_t^u$.

Let $u^*$ and $v$ be  two admissible controls.  Denote $\bar u=v-u^*$.    For any $\e\in \RR$, we denote
$u^\e=u^*+\e \bar u$.   Then  $\bar u$ and $u^\e$ are also admissible controls.  Corresponding to
$u^*$ and $u^\e$  there are   solutions
  $x^\e(\cdot)=x(\cdot; u^\e(\cdot))$ and $x^*(\cdot)=x(\cdot;
u^*(\cdot))$ to the  equations (\ref{5.3.2}) .  That is
\begin{eqnarray}
x^*(t)
&=&x_0+\int_0^tb(s,x^*(s),u^*(s))ds+\sum_{j=1}^m \int_0^t\sigma_j(s,x^*(s),u^*(s))dB_j^H (s),\nonumber\\
\label{e.4.4}\\
x^\e(t)&=&x_0+\int_0^tb(s,x^\e(s),u^\e(s))ds+\sum_{j=1}^m
\int_0^t\sigma_j(s,x^\e(s),u^\e(s))dB_j^H(s)\,. \nonumber\\
\label{e.4.5}
\end{eqnarray}
To obtain our results of maximum  principle, we need the following.
\begin{proposition}\label{l.3.1}   Assume (H4) and (H5).
Let $x^\e(\cdot)$ and $x^*(\cdot)$ be the solutions of equation  (\ref{5.3.2}) corresponding to $u^\e(\cdot)$ and $u^*(\cdot)$ respectively. Then
\begin{eqnarray*}
\lim_{\vare\rightarrow 0} \sup\limits_{0\le t\le T} |x^\e(t)-x^*(t)| &=& 0\,, \\
\lim_{\vare\rightarrow 0}  \sup\limits_{0\le t\le T} \left
| \frac{x^\e(t)-x^*(t)}{\e}-y(t) \right|&=&0
\end{eqnarray*}
almost surely as $\e \to 0$, where $y(t)$ is the solution of the following
equation.
\begin{eqnarray} \label{5.3.6}
y(t)&=&\int_0^t\left[ b_x^*(s)y(s)+b_u^*(s)\bar
u(s)\right]ds\nonumber\\
&&\qquad +\int_0^t\sum\limits_{j=1}^m\left[\sigma_x^{j*}(s)y(s)+\sigma_u^{j*}(s)\bar
u(s)\right]  d^\circ B_j^H(s)\,.
\end{eqnarray}
\end{proposition}

\medskip\noindent{\it Proof of Proposition \ref{l.3.1}}  \quad
 We shall follow the idea of \cite{HN1} to prove the above lemma.
Fix  $\frac12>\al>0$ and  $\be>0$ such that  $0<1-\alpha<\beta<H$.
Let $g_1(t)\,, \cdots\,, g_m(t)$ be any given  functions of $\beta$-
H\"older continuous. Consider the following   deterministic
differential equation
\begin{equation}\label{5.3.7}
x(t)=x_0+\int_0^tb(s,x(s),u(s))ds+\sum_{j=1}^m \int_0^t\sigma_j(s,x(s),u(s))dg_j(s)\,.
\end{equation}
Corresponding to  the admissible controls $u^*$ and $u^\e$,   the above equation \eref{5.3.7} has also
two solutions, still denoted by $x^*$ and $x^\e$.
Corresponding to $u^\e$, we shall obtain  an equations
depending  on a parameter $\e$.   We can consider
more general one:
\begin{equation}\label{e.4.8a}
x^\e(t)=x_0+\int_0^tb(\e, s,x^\e(s))ds+\sum_{j=1}^m
\int_0^t\sigma_j(\e, s, x^\e(s) )dg_j (s)\,.
\end{equation}

\begin{lemma}
\label{l.4.2} Let  \ $b(\e, t, \cdot):   \RR^n\rightarrow \RR^n$
and $\si_j(\e, t, \cdot):    \RR^n\rightarrow \RR^n$, $j=1, \cdots,
m$,
 be a continuously differentiable functions  with uniformly bounded
 derivatives.  Let  $\si_j(\e, \cdot, x):
[0, T]\rightarrow \RR^n$, $j=1, \cdots, m$,  be uniformly H\"older
continuous with exponent $\ga >\al $.  That means  that there is a
constant $M\in (0, \infty)$, independent of $\e, t$, and $x$,  such
that
\[
\left| \frac{\partial}{\partial x_i} b_j(\e, t, x)\right| \le M \,,
\quad \left| \frac{\partial}{\partial x_i} \si_{j} (\e, t, x)
\right|\le M
\]
and
\[
\left|   \si_{j} (\e, t, x)-\si_{j} (\e, s, x)   \right|\le
M|t-s|^{\ga }\,.
\]
Assume also that $b$ and $\si_j$ satisfy the following   uniform linear growth condition:
\[
\left|  b_j(\e, t, x)\right| \le M (1+|x|) \,,
\quad \left| \  \si_{j} (\e, t, x)
\right|\le   M (1+|x|)\,.
\]
Then there are constants $C$ and $c$ independent of  $\e $, $M$, and $g$, such that for
all $T$,
\begin{equation}
\sup_{0\leq t\leq T}|x^\vare ({t}) | \
\leq Ce^{c\Vert g\Vert _{0,T,{\beta }}^{ \frac{1}{\be}}} (|x_{0}|+1)
  \label{e.4.9a}
\end{equation}
and
\begin{equation}
\|x^\vare \|_{0, T, \beta}
\leq Ce^{c\Vert g\Vert _{0,T,{\beta }}^{ \frac{1}{\be}}} (|x_{0}|+1)
\,.  \label{e.4.9aa}
\end{equation}
\end{lemma}
\begin{remark}
We also have that $\EE \sup_{0\le t\le T}|x^\e(t)|^p$ and $\EE \|x^\e\|_{0,T,\be}^p$ for any $p>0$ are uniformly bounded (independent of $\e$), when $g=B^H$ and $\frac12<\be<H$, by Fernique theorem.  Actually, since $\frac1\be<2,$ then by Fernique theorem, we have $\EE e^{p\|B^H\|^{\frac1\be}_{0,T,\be} }<\infty$ for all $p>0.$
\end{remark}
\begin{proof}  The existence of the solution $x^\vare$ for every $\vare$ is known.  See example in
\cite{HN1} and \cite{NR}.  We shall sketch the proof of the bounds \eref{e.4.9a}
following  idea in the proof of Theorem 3.1  of
\cite{HN1}.  Without loss of generality we assume that $n=m=1$.
Set $\ $ $\Vert g\Vert _{{\beta }}=\Vert g\Vert _{0,T,{\beta }}$. We
can assume that $\Vert g\Vert _{\beta }>0$, otherwise the
inequalities are obvious.

\smallskip
\noindent {\bf Step 1.}  From \eref{e.2.5}, we have for any $0<s<t<T$
\begin{eqnarray}
&& \left|\int_{s}^{t}\si(\e, r, x^\vare(r))dg_{r}\right|
 \leq   C \|g\|_{\be}
\left[\int_s^t \frac{\left| \si(\e, r, x^\vare(r))\right|  }{(r-s)^\al}
(t-r)^{\al+\be-1} dr\right. \nonumber\\
&&\quad\qquad  \left. + \int_s^t \int_s^r \frac{|\si(\e, r, x^\vare(r))-\si(\e, \tau, x^\vare(\tau))|}{
|r-\tau|^{\al+1}} (t-r)^{\al+\be-1} d\tau dr\right]\nonumber\\
&=:&I_1+I_2\,,\label{e.4.10a}
\end{eqnarray}
where and in what follows, $C$ is a universal constant (independent $g$ and $M$).
It is easy to see from the assumption of the lemma that
\begin{equation}
I_1\le CM \|g\|_\be \left[ 1+\|x^\vare \|_{ s, t, \infty}\right]  (t-s)^{\be}\,.
\label{e.4.11a}
\end{equation}
$I_2$ can be estimated as
\begin{eqnarray}
I_2
&\leq & C \|g\|_{\be}
\left[    \int_s^t \int_s^r \frac{ |\si(\e, r, x^\vare(r))-\si(\e, \tau, x^\vare (r))|}{
|r-\tau|^{\al+1}} (t-r)^{\al+\be-1} d\tau dr\right.\nonumber \\
&&\quad\left.
  \int_s^t \int_s^r \frac{|\si(\e, \tau, x^\vare(r))-\si(\e, \tau, x^\vare(\tau))|}{
|r-\tau|^{\al+1}} (t-r)^{\al+\be-1} d\tau dr\right] \nonumber\\
&\le& CM  \|g\|_{\be} \left[ (t-s)^\ga +\|x^\vare\|_{s, t, \be} (t-s)^\be \right] (t-s)^{ \be }\,.
\label{e.4.12a}
\end{eqnarray}
On the other hand
\begin{equation}
\left| \int_s^t b(\e, r, x^\vare(r) ) dr\right|\le CM \left(1+\|x^\vare\|_{s, t, \infty} \right)
(t-s)\,.\label{e.4.13a}
\end{equation}
Therefore from \eref{e.4.10a}-\eref{e.4.13a} we see that  the solution $x^\vare$ to \eref{e.4.8a} satisfies
\begin{equation}
\Vert x^\vare\Vert _{s,t,{\beta }}\leq C  M  \left[
1+ \left\Vert x^\vare \right\Vert _{s,t,\infty
}\right] + C  M \Vert g\Vert _{{\beta }}\left[
1+ \left\Vert x^\vare \right\Vert _{s,t,\infty
}+ \Vert x^\vare \Vert _{s,t,{\beta }%
}(t-s)^{{\beta }}\right], \,\text{for }s,t\in[0,T]\,.\label{e.4.14a}
\end{equation}

\smallskip
\noindent {\bf Step 2.}
Choose $\Delta $ such that%
\begin{equation}
\Delta = \left( \frac{1}{3C M [1+\Vert g\Vert _{{\beta }}]}\right) ^{\frac{1}{\beta }}. \label{g6}
\end{equation}
Then, for all $s$ and $t$ such that $0\le t-s\leq \Delta $ we have
\begin{equation}
\Vert x^\vare \Vert _{s,t,{\beta }}\leq \frac{3}{2}C M [1+\Vert g\Vert _{{\beta }}] \left(
1+ \left\Vert x^\vare \right\Vert _{s,t,\infty
}\right) .\label{e0}
\end{equation}
Therefore, when $0\le t-s\leq \Delta $
\begin{equation}
|x^\vare({t})| \leq |x^\vare ({s})|+ \frac{3}{2}C M [1+\Vert g\Vert _{{\beta }}] \left(
1+ \left\Vert x^\vare \right\Vert _{s,t,\infty
}\right) \Delta ^{\beta },  \label{e1}
\end{equation}%
or
\[
\|x^\vare\|_{s,t, \infty} \leq |x^\vare ({s})|+ \frac{3}{2}C M [1+\Vert g\Vert _{{\beta }}] \left(
1+ \left\Vert x^\vare \right\Vert _{s,t,\infty
}\right) \Delta ^{\beta }
\]%
for $0\le t-s\leq \Delta $.
Using again (\ref{g6}) we get
\[
\left\Vert x^\vare \right\Vert _{s,t,\infty }\leq 2|x^\vare({s})|+3C M[1+\Vert g\Vert _{{\beta }
}]  {\Delta }^{\beta }.
\]
Since \eref{g6} implies
\[
\Delta \leq \left( \frac{2}{3C M [1+\Vert g\Vert _{{%
\beta }}] }\right) ^{\frac{1}{\beta }}.
\]%
Then%
\[
\left\Vert x^\vare \right\Vert _{s,t,\infty }\leq 2\left( |x^\vare ({s})|+1\right) .
\]
Hence,%
\begin{equation} \label{a1}
\sup_{0\leq r\leq t}|x^\vare({r})|\leq 2\left( \sup_{0\leq r\leq s}|x^\vare({r}) |+1\right)
\quad \forall\    t-s\le \Delta\,, 0\le s\le t\le T \,.
\end{equation}
Now we divide   the interval $[0,T]$ into $n=[T/\Delta ]+1$ subintervals,
and use the estimate (\ref{a1}) in every interval to obtain%
\[
\sup_{0\leq t\leq T}|x^\vare({t})| \leq 2^{n}\left( |x^\vare (0)|+1\right)
 \leq 2^{\frac{T}{\Delta}+1}\left( |x^\vare (0)|+1\right) \ .
\]
Finally, we have  from \eref{g6}
\[
\sup_{0\leq t\leq T}|x^\vare ({t})| \
\leq Ce^{c \Vert g\Vert _{0,T,{\beta }}^{ \frac{1}{\be}}} (|x^\vare (0) |+1)\,.
\]

\smallskip
\noindent {\bf Step 3.}  From Equation \eref{e0}, we see also when $t-s\le \De$ and $0\le s<t\le T$
\begin{equation}
\|x^\vare \|_{s, t, \be}  \
\leq Ce^{c  T M^{ \frac{1}{\be}}\Vert g\Vert _{0,T,{\beta }}^{ \frac{1}{\be}}} (|x^\vare (0)|+1)
\label{e3}
\end{equation}
since $x\le Ce^{cx^{1/\be}}\,,\forall \  x\ge 0$ with  some constants $c$ and $C$.
For general $0\le s<t\le T$, we denote $s=t_0<t_1< \cdots<t_{n-1}<t_n=t$ so that $t_{k}-t_{k-1}\le \De$.
Then
\begin{eqnarray*}
\left|\frac{x^\vare(t)-x^\vare(s)}{(t-s)^\beta}\right|
&\le&\sum_{k=0}^{n-1} \left|\frac{x^\vare(t_{k+1})-x^\vare(t_k)}{(t_{k+1}-t_k)^\be}\right|\\
&\le& \sum_{k=0}^{n-1} Ce^{c  \Vert g\Vert _{0,T,{\beta }}^{ \frac{1}{\be}}} (|x^\vare (0)|+1)\\
&\le& n Ce^{c  \Vert g\Vert _{0,T,{\beta }}^{ \frac{1}{\be}}} (|x^\vare (0) |+1)\,.
\end{eqnarray*}
With the same argument as \eref{e3}, we have \eref{e.4.9aa}.
The proof of the theorem is now complete.
\end{proof}

It is clear that the Proposition \ref{l.3.1}  is a consequence of
the following lemma.

\begin{lemma}\label{l.4.3}  Assume  (H4) and (H5)  and assume $0<1-\mu<1-\alpha<\beta<H$.   Let
 $g_1(t), \cdots, g_m(t)$ be  $\beta$-H\"older continuous  functions  of  $t\in [0,T]$.  Let $x^\e(\cdot)$ and $x^*(\cdot)$ be the solutions of equations (\ref{5.3.7}) corresponding to $u^\e(\cdot)$ and $u^*(\cdot)$,  respectively. Then
\begin{eqnarray*}
&\lim_{\e\rightarrow 0} \sup\limits_{0\le t\le T}|x^\e(t)-x^*(t)|= 0\,,\\
& \lim_{\e\to0}\|x^\e-x^*\|_{1-\alpha} =  0\,,\\
&\lim_{\e\rightarrow 0} \ \sup\limits_{0\le t\le T}\left
| \frac{x^\e(t)-x^*(t)}{\e} -y(t) \right|= 0\,,
\end{eqnarray*}
 where $y(t)$ is the solution of the following linear equation.
\begin{equation}\label{e-y}
y(t)=\int_0^t[b_x^*(s)y(s)+b_u^*(s)\bar
u(s)]ds+\int_0^t\sum\limits_{j=1}^m[\sigma_x^{j,*}(s)y(s)+\sigma_u^{j,*}(s)\bar
u(s)]dg_j(s).
\end{equation}
Moreover, the above limits hold in $L^p$ sense as well for all $p>0.$
\end{lemma}

\begin{proof}   We   follow the idea of \cite{HN1}.  We divide the proof into several steps.

\medskip
\noindent\textbf{Step 1}.\quad
To simplify the notation, we assume $n=d=m=1$.  The general case  only increases the notational complexity.    Throughout this paper we shall use $C$ to denote a generic constant, independent of
$\e$,  whose values may be different in different occurrences.
Denote  $\sigma(\cdot, x^\e(\cdot),u^\e(\cdot))$ and $\sigma(\cdot,
x^*(\cdot),u^*(\cdot))$ by $\sigma^\e(\cdot)$ and $\sigma^*(\cdot)$
respectively.  Set   $y^\e(\cdot)=x^\e(\cdot)-x^*(\cdot)$.  We have that
\begin{eqnarray}
 |y^\e(t)-y^\e(s)|
&\le &\left |\int_s^t\left(b^\e(r)-b^*(r)\right)dr \right |+\left |\int_s^t[\sigma^\e(r)-\sigma^*(r)]dg(r) \right | \nonumber\\
&\le&  L \int_s^t|y^\e(r)|dr+\e L \int_s^t|\bar u(r)|dr\nonumber\\
&&\qquad
                +k\|g\|_{\beta}\int_s^t(t-r)^{\alpha+\beta-1}\left[\cD_{s+}^{\alpha}(\sigma^\e(r)-\sigma^*(r)\right]dr\nonumber \\
 &\le &   L \int_s^t|y^\e(r)|dr+\e L \int_s^t|\bar u(r)|dr +I_1+I_2\,, \label{e.4.6}
\end{eqnarray}
where
\begin{eqnarray*}
I_1 &:=& k\|g\|_{\beta}\int_s^t(t-r)^{\alpha+\beta-1}\frac{|\sigma^\e(r)-\sigma^*(r)|}{(r-s)^\alpha}dr\\
\\
I_2&:=& k\|g\|_{\beta}\int_s^t\int_s^r(t-r)^{\alpha+\beta-1}\frac{|\sigma^\e(r)-\sigma^*(r)-[\sigma^\e(\tau)-\sigma^*(\tau)]|}{(r-\tau)^{\alpha+1}}d\tau dr\,.
\end{eqnarray*}
$I_1$  is handled easily.
\begin{eqnarray}
I_1 &\le&
 C \|g\|_\beta\int_s^t\frac{|y^\e(r)|(t-r)^{\alpha+\beta-1}}{(r-s)^\alpha}dr+C
 \e  \|g\|_\beta\int_s^t\frac{|\bar u(r)|(t-r)^{\alpha+\beta-1}}{(r-s)^\alpha}dr\,. \nonumber\\
 &\le&
 C \|g\|_\beta (t-s)^{\be }   \sup_{s\le r\le t} |y^\e(r)| +C
 \e  \|g\|_\beta (t-s)^{\be }      \,,
 \label{e.4.7}
\end{eqnarray}

\medskip
\noindent\textbf{Step 2}.\quad
To estimate  $I_2$,  let us consider the integral in $I_2$, denoted by $\tilde I_2$. .
\begin{eqnarray*}
\tilde I_2 & :=&\int_s^t\int_s^r\frac{(t-r)^{\alpha+\beta-1}}
{(r-\tau)^{\alpha+1}}\left|\sigma^\e(r)-\sigma^*(r)
-[\sigma^\e(\tau)-\sigma^*(\tau)]\right|d\tau dr\\
&=&\int_s^t\int_s^r\frac{(t-r)^{\alpha+\beta-1}}
{(r-\tau)^{\alpha+1}}\left|\int_0^1\sigma_x(r,x^*(r)
+\lambda y^\e(r),u^\e(r))d\lambda\cdot y^\e(r)\right.\\
&&+\int_0^1\sigma_u(r,x^*(r),u^*(r)+\lambda(u^\e(r)-u^*(r)))
d\lambda\cdot (u^\e(r)-u^*(r))\\
&&-\int_0^1\sigma_x(\tau,x^*(\tau)+\lambda y^\e(\tau),u^\e(\tau))
d\lambda\cdot y^\e(\tau)\\
&&-\left.\int_0^1\sigma_u(\tau,x^*(\tau),u^*(\tau)
+\lambda(u^\e(\tau)-u^*(\tau)))d\lambda\cdot (u^\e(\tau)-u^*(\tau))\right| d\tau dr\\
&=&\int_s^t\int_s^r\frac{(t-r)^{\alpha+\beta-1}}
{(r-\tau)^{\alpha+1}}\left|\int_0^1\sigma_x(r,x^*(r)+\lambda y^\e(r),u^\e(r))
d\lambda\cdot [y^\e(r)-y^\e(\tau)]\right.\\
&&+\int_0^1[\sigma_x(r,x^*(r)+\lambda y^\e(r),u^\e(r))-\sigma_x(\tau,x^*(\tau)
+\lambda y^\e(\tau),u^\e(\tau))]d\lambda\cdot y^\e(\tau)\\
&&+\int_0^1\sigma_u(r,x^*(r),u^*(r)+\lambda \e\bar u(r))d\lambda\cdot
\e(\bar u(r)-\bar u (\tau))\\
&&+\left.\int_0^1[\sigma_u(r,x^*(r),u^*(r)+\lambda \e \bar u(r))
-\sigma_u(\tau,x^*(\tau),u^*(\tau)+\lambda \e\bar u(\tau))]
d\lambda\cdot \e\bar u(\tau)\right| d\tau dr\,.
\end{eqnarray*}
Since $\si_x$ and $\si_u$  are bounded,  we have
 \begin{eqnarray*}
\tilde I_2 &\le &
C\int_s^t\int_s^r\frac{(t-r)^{\alpha+\beta-1}}
{(r-\tau)^{\alpha+1}} |y^\e(r)-y^\e(\tau)|d\tau dr \\
& &+  C\e
 \int_s^t\int_s^r\frac{(t-r)^{\alpha+\beta-1}}{(r-\tau)^{\alpha+1}}
|\bar u(r)-\bar u(\tau)|d\tau dr\\
 &&+ \int_s^t\int_s^r\frac{(t-r)^{\alpha+\beta-1}}
 {(r-\tau)^{\alpha+1}}\int_0^1\bigg[ \left|\sigma_x(r,x^*(r)+\lambda
 y^\e(r),u^\e(r))-\sigma_x(\tau,x^*(r)+\lambda y^\e(r),u^\e(r))\right|\\
 &&+\left|\sigma_x(\tau,x^*(r)+\lambda y^\e(r),u^\e(r))
 -\sigma_x(\tau,x^*(\tau)+\lambda
 y^\e(\tau),u^\e(\tau))\right|\bigg]
 d\lambda \cdot |y^\e(\tau)|d\tau dr\\
 &&+\int_s^t\int_s^r\frac{(t-r)^{\alpha+\beta-1}}{(r-\tau)^{\alpha+1}}
 \int_0^1\bigg[ \left|\sigma_u(r,x^*(r),u^*(r)+\lambda \e \bar u(r))
 -\sigma_u(\tau,x^*(r),u^*(r)+\lambda \e \bar u(r))\right| \\
& &+\left|\sigma_u(\tau,x^*(r),u^*(r)+\lambda \e \bar u(r))
-\sigma_u(\tau,x^*(\tau),u^*(\tau)+\lambda \e \bar
u(\tau))\right|\bigg]  d\lambda \cdot \e|\bar u(\tau)|d\tau dr\,.
 \end{eqnarray*}
Again since the second derivatives  $\si_{xx}$, $\si_{uu}$,   and
$\si_{xu}$  are bounded and the first derivatives $\si_x$ and
$\si_u$ are $\gamma$-H\"older continuous in $t$,  we have
 \begin{eqnarray*}
 \tilde I_2
 &\le & C \int_s^t\int_s^r\frac{(t-r)^{\alpha+\beta-1}}
 {(r-\tau)^{\alpha+1}}|y^\e(r)-y^\e(\tau)|d\tau dr\\
 &&\qquad + C \e
 \int_s^t\int_s^r\frac{(t-r)^{\alpha+\beta-1}}{(r-\tau)^{\alpha+1}}
 |\bar u(r)-\bar u(\tau)|d\tau dr\\
 &&+C  \int_s^t\int_s^r\frac{(t-r)^{\alpha+\beta-1}}
 {(r-\tau)^{\alpha+1}}\cdot \left[(r-\tau)^\gamma +|x^*(r)-x^*(\tau)| \right.\\
 &&\left.+|x^\e(r)-x^\e(\tau)| +|u^\e(r)-u^\e(\tau)|  \right]
 \cdot |y^\e(\tau)|d\tau dr\\
 &&+C\e    \int_s^t\int_s^r\frac{(t-r)^{\alpha+\beta-1}}{(r-\tau)^{\alpha+1}}
 \cdot \left[(r-\tau)^\gamma +|x^*(r)-x^*(\tau)| \right.\\
 &&\left.+|u^*(r)-u^*(\tau)| +|u^\e(r)-u^\e(\tau)|  \right]\cdot
 |\bar u(\tau)|d\tau dr\,.
 \end{eqnarray*}
 Since both $u$ and $v$ are admissible controls, they are H\"older
 continuous of order $(\mu$.  Thus
 $u^\e $ is   uniformly  H\"older continuous of order $\mu$.
 From Lemma \ref{l.4.2},  we know that $x^\e$ is also uniformly
 H\"older continuous of order
 $\be$ and hence   $x^\e$ is also uniformly
 H\"older continuous of order
 $1-\al$
 %

  We also use the boundedness of $u$.  Hence,   we have
\begin{eqnarray*}
 \tilde I_2
 &\le &  C
\int_s^t\int_s^r\frac{(t-r)^{\alpha+\beta-1}}{(r-\tau)^{\alpha+1}}
|y^\e(r)-y^\e(\tau)|d\tau dr\\
&&+C
\int_s^t\int_s^r\frac{(t-r)^{\alpha+\beta-1}}{(r-\tau)^{\alpha+1}}
[(r-\tau)^\gamma+ (r-\tau)^{(1-\alpha) }+ (r-\tau)^{\mu }]
|y^\e(\tau)|d\tau dr\\
&&+C\e
\int_s^t\int_s^r\frac{(t-r)^{\alpha+\beta-1}}{(r-\tau)^{\alpha+1}}
[(r-\tau)^\gamma+ (r-\tau)^{(1-\alpha) }+ (r-\tau)^{\mu }]
 d\tau dr\\
 &\le &  C
\int_s^t\int_s^r\frac{(t-r)^{\alpha+\beta-1}}{(r-\tau)^{2\alpha }}
\|y^\vare \|_{1-\al, \tau, r} d\tau dr\\
&&+C
\int_s^t\int_s^r\frac{(t-r)^{\alpha+\beta-1}}{(r-\tau)^{\alpha+1}}
[(r-\tau)^\gamma+ (r-\tau)^{(1-\alpha) }+ (r-\tau)^{\mu }]
|y^\e(\tau)|d\tau dr\\
&&+C\e
\int_s^t\int_s^r\frac{(t-r)^{\alpha+\beta-1}}{(r-\tau)^{\alpha+1}}
[(r-\tau)^\gamma+ (r-\tau)^{(1-\alpha) }+ (r-\tau)^{\mu }]
 d\tau dr\,.
\end{eqnarray*}
Denote $\nu=(1-\al)\wedge \gamma\wedge \mu$.  The above inequality
gives
\begin{eqnarray*}
 \tilde I_2
 &\le &  C (t-s)^{1+\be-\al}
\|y^\vare \|_{1-\al, s, t}  +C (t-s)^{\be+\nu }\|y^\vare \|_{s, t, \infty}+
C\e (t-s)^{\be+\nu }
\end{eqnarray*}
Therefore, we have
\begin{equation}
I_2\le C \|g\|_{\beta}\left[ (t-s)^{1+\be-\al}
\|y^\vare \|_{1-\al, s, t}  +  (t-s)^{\be+\nu }\|y^\vare \|_{s, t, \infty}+
 \e (t-s)^{\be+\nu }\right]\,.\label{e.4.8}
\end{equation}

Combining \eref{e.4.7} and \eref{e.4.8},  we have
\begin{eqnarray}
 |y^\e(t)-y^\e(s)|
&\le&  C \|g\|_{\beta} \bigg[ (t-s)^{1+\be-\al}
\|y^\vare \|_{1-\al, s, t}  +  (t-s)^{\be  }\|y^\vare \|_{s, t, \infty}\nonumber\\
&&\qquad +
 \e (t-s)^{\be  }\bigg]  \,.\label{e.4.9}
\end{eqnarray}
Thus we have
\begin{eqnarray}
\|y^\e\|_{1-\alpha, s, t}
&\le& C  \|g\|_{\beta} \bigg[ \|y^\e\|_{1-\alpha, s, t}(t-s)^\beta + (  t-s)^{\beta+\alpha-1}
\|y^\e\|_{ s,t,\infty }\nonumber\\
&&\qquad   + \e
 (t-s)^{ \beta+\alpha-1} \bigg] \,.\label{e.4.10}
\end{eqnarray}

\medskip
\noindent\textbf{Step 3}.\quad
We choose  $\Delta_1$ such that
$$
C   \|g\|_{\beta}  \Delta_1^\beta=\frac12\,.
$$
From  the above equation \eref{e.4.10}   it follows that  if $0<t-s\le \De_1$, then
\begin{equation*}
\|y^\e\|_{1-\alpha, s, t}\le  C \|g\|_{\beta}   (t-s)^{\al+\be-1}  \|y^\e\|_{s, t, \infty} +C  \e \|g\|_{\beta}  (t-s)^{\al+\be-1}  \,.
\end{equation*}
Since $|y(t)|\le |y(s)|+|t-s| ^{1-\al} \|y\|_{1-\alpha, s, t}$,   the above inequality yields
\begin{equation*}
|y^\e(t)|\le  |y^\e(s)|+ C \|g\|_{\beta}   (t-s)^{ \be }  \|y^\e\|_{s, t, \infty} +C  \e \|g\|_{\beta}  (t-s)^{\ \be } \,,    \quad \forall \  0<t-s\le \De_1\,.
\end{equation*}
which implies easily
\begin{equation}
\|y^\e\|_{s, t, \infty}  \le  |y^\e(s)|+ C \|g\|_{\beta}   (t-s)^{ \be }  \|y^\e\|_{s, t, \infty} +C  \e \|g\|_{\beta}  (t-s)^{\ \be } \,,    \quad \forall \  0<t-s\le \De_1\,. \label{e.4.11}
\end{equation}
Now   we choose $\Delta_2 $ such that
$$
C \|g\|_{\beta}   \Delta_2 ^{ \be }  =1/2\,.
$$
[Notice  that the constant $C$ may be different than that in the definition for $\De_1$.]\  \
Then for all $ 0\le s<t\le T, t-s\le \Delta_0:=\Delta_1\wedge
\Delta_2$,   we have from equation \eref{e.4.11},
\begin{equation}
|y^\e|_{ \infty, s,t}\le  2|y^\e(s)|+  C \e
 \|g\|_{\beta}  (t-s)^\be \,. \label{e.4.12}
\end{equation}
We apply the above inequality to $s=0$ and $t-s\le \De_0$ and notice that $y^\e(0)=0$.  We have that
\begin{equation*}
|y^\e|_{0, \De_0, \infty}\le   C \e
 \|g\|_{\beta}   \De _0^\be \,.
\end{equation*}
In general,  for any integer positive $k$ such that $k\De_0< T$,  if we let $s=  \De_0$ and
$t\in [k\De_0, (k+1) \De_0]$, then we have
\begin{equation*}
|y^\e|_{0, (k+1)\De_0, \infty}\le  2|y_{k\De_0} ^\e| +   C \e
 \|g\|_{\beta}   \De _0^\be \,.
\end{equation*}
This implies
\begin{equation}
|y^\e|_{0, k \De_0, \infty}\le      C (2^{k}-1) \e
 \|g\|_{\beta}   \De _0^\be \,. \label{e.4.13}
\end{equation}
%
In the equation \eref{e.4.13},  if we let
\[
k=[T/\De_0]+1\le 2T/\De_0 \le 2T/\De_1+2T/\De_2=C T \|g\|_\be^{1/\be} \,.
\]
Then \eref{e.4.13} yields
\begin{equation}
|y^\e|_{0, T, \infty}\le      C  2^{CT\|g\|_\be^{1/\be} }  \e
     \,. \label{e.4.14}
\end{equation}
Therefore  we obtain
\begin{equation}\label{e.4.15}
\lim_{\e\rightarrow 0}  \sup\limits_{0\le t\le T}|x^\e(t)-x^*(t)|= 0\,.
\end{equation}
In the same way as in Step 3 in Lemma \ref{l.4.2}, we can also prove
\[\lim_{\e\to0}\|x^\e-x^*\|_{1-\alpha}=0.\]

%

\medskip
\noindent\textbf{Step 4}.\quad
Denote  $\eta^\e(t)=\frac{1}{\e}y^\e(t)-y(t)$.  Then  we can write for all $0\le t\le T$,
\begin{align*}
\eta^\e(t)=&\frac{1}{\e}\int_0^t[b^\e(r)-b^*(r)-\e (b_x^*(r)y(r)+b_u^*(r)\bar u(r))]dr\\
& +\frac{1}{\e}\int_0^t[\sigma^\e(r)-\sigma^*(r)-\e
(\sigma_x^{*}(r)y(r)+\sigma_u^{*}(r)\bar u(r))]dg(r)\,.
\end{align*}
Hence we have
\begin{eqnarray}
 |\eta^\e(t)-\eta^\e(s)|
&\le & I_3+I_4\,,\label{e.4.16}
\end{eqnarray}
where
\begin{eqnarray*}
I_3&=& \left |\frac1{\e}\int_s^t[b^\e(r)-b^*(r)-\e (b_x^*(r)y(r)+b_u^*(r)\bar u(r))]dr \right |\,, \\
I_4 &=& \left |\frac1{\e}\int_s^t[\sigma^\e(r)-\sigma^*(r)-\e (\sigma_x^{*}(r)y(r)+\sigma_u^{*}(r)\bar u(r))]dg(r) \right |\,.
\end{eqnarray*}
 Using the argument as for $I_1$ and from the boundedness and  the Lipschitz continuity of the derivative
  $b_x$ and $b_u$,  and  from
the inequality  \eref{e.4.14},    we have  for $0\le s<t\le T$,
\begin{eqnarray}
I_3
&=&  \bigg |\frac1{\e}\int_s^t \int_0^1 \bigg\{ \left[
 b_x (r, x^*(r)  +\la y^\e (r)\,,  u^\e (r)  )  y^\e (r) - \e b_x^*(r)y(r)\right] \nonumber\\
 &&\qquad   +\e \left[
 b_u (r, x^*(r)  \,,  u^* (r)+\e \la \bar u(r)  )  \bar u  (r) -   b_u^*(r)\bar u(r)\right] \bigg\} d\la dr
     \bigg|\nonumber \\
&\le&   \int_s^t \int_0^1 \bigg\{ \left|
 b_x (r, x^*(r)  +\la y^\e (r)\,,  u^\e (r)  ) \right|\left| \eta^\e (r)     \right| \nonumber \\
 &&\quad +  \left|
 b_x (r, x^*(r)  +\la y^\e (r)\,,  u^\e (r)  )-b_x^*(r)  \right|\left|  y (r)     \right| \nonumber \\
 &&\qquad   +
 \left|
 b_u (r, x^*(r) \,,  u^*(r)  +\e \la \bar u  (r)  )   -   b_u^*(r)\right|\left| \bar u(r)\right| \bigg\} d\la dr
     \bigg|\nonumber \\
&\le &  C  \int_s^t|\eta^\e(r)|dr+ L \int_s^t\left[ \left( |y^\e(r)|+\e |\bar u(r)|
            \right)   \cdot | y(r)|\right] dr+\e L \int_s^t|\bar u(r)|^2dr\nonumber
\end{eqnarray}
Since $\sup_{0\le r\le T} |y^\e|\le C\e$  we obtain
\begin{eqnarray}
I_3
&\le&     C  (t-s) \left[ \|\eta^\e \|_{s, t, \infty}+\e\right]  \,.\label{e.4.17}
\end{eqnarray}

\medskip
\noindent\textbf{Step 5}.\quad
To estimate $I_4$ we shall use   the consequence \eref{e.2.5} of
the  fractional integration by parts formula \eref{e.2.4}.
Denote
$$
\bar \sigma_\e (\cdot)=\sigma^\e(\cdot)-\sigma^*(\cdot)-\e
(\sigma_x^{*}(\cdot)y(\cdot)+\sigma_u^{*}(\cdot)\bar u(\cdot)).
$$
From \eref{e.2.5}    it follows
\begin{eqnarray*}
I_4
&\le &C\|g\|_{\beta}\frac{1}{\e}\int_s^t\frac{|\bar\sigma_\e(r)|  }
{(t-r)^{ 1-\al-\be }(r-s)^\alpha}dr\\
             &&+ C\|g\|_{\beta}\frac{1}{\e}\int_s^t\int_s^r
             \frac{|\bar \sigma_\e(r)-\bar\sigma_\e(\tau)|}{(r-\tau)
             ^{\alpha+1}(t-r)^{1-\alpha-\beta}}d\tau dr\,.
\end{eqnarray*}
Use the same technique as for $I_3$ and $I_2$ to obtain
\begin{eqnarray}
I_4
&\le &C\|g\|_{\beta} \left(I_{41}+I_{42}+I_{43}+I_{44}\right)\,,
\label{e.4.18}
\end{eqnarray}
where
\begin{eqnarray*}
I_{41}&:=&
 \int_s^t\frac{(t-r)^{\alpha+\beta-1}}{(r-s)^\alpha}
\left|\int_0^1\sigma_x(r,x^*(r)+\lambda y^\e(r),u^\e(r)
)d\lambda \cdot \eta^\e(r)\right.\\
&&+\int_0^1\left(\sigma_x(r,x^*(r)+\lambda y^\e(r),u^\e(r))-\sigma_x^*(r)\right)d\lambda \cdot y(r)\\
&&+\left.\int_0^1\left(\sigma_u(r,x^*(r) ,u^*(r)
+\lambda \e\bar u(r))-\sigma^*_u(r)\right)d\lambda \cdot \bar u(r)\right|dr\\
I_{42} &=&  \int_s^t\int_s^r\frac{(t-r)^{  \al+\be -1 }}
 {(r-\tau)^{\alpha+1}}\left|\int_0^1\sigma_x(r,x^*(r)
 +\lambda y^\e(r), u^\e(r)) d\lambda \cdot \eta^\e(r)\right.\\
 &&\left. -\int_0^1\sigma_x(\tau,x^*(\tau)+\lambda y^\e(\tau),u^\e(\tau)
)d\lambda \cdot \eta^\e(\tau)\right|d\tau dr\\
I_{43}&:=&  \int_s^t\int_s^r\frac{(t-r)^{  \al+\be -1 }}
 {(r-\tau)^{\alpha+1}}\left| \int_0^1\left(\sigma_x(r,x^*(r)+\lambda y^\e(r), u^\e(r))-\sigma^*_x(r)\right)d\lambda \cdot y^\e(r)\right. \\
&&\left. -\int_0^1\left(\sigma_x(\tau,x^*(\tau)+\lambda y^\e(\tau),u^\e(\tau)
)-\sigma^*_x(\tau)\right)d\lambda \cdot y^\e(\tau)\right|d\tau dr\\
I_{44}&:=& \int_s^t\int_s^r\frac{(t-r)^{  \al+\be -1 }}
 {(r-\tau)^{\alpha+1}}\left|\int_0^1\left(\sigma_u(r,x^*(r), u^*(r)
 +\lambda \e \bar u(r)  )-\sigma^*_u(r)\right)d\lambda \cdot (\e\bar u(r))\right. \\
 &&\left.-\int_0^1\left(\sigma_u(\tau,x^*(\tau), u^*(\tau)
 +\lambda \e \bar u(\tau))-\sigma^*_u(\tau)\right)d\lambda \cdot(\e \bar u(\tau))\right|d\tau dr\,.
\end{eqnarray*}
We shall  use the boundedness of the first derivatives of $\si$.
Since the second derivatives of $\si$ with respect to $x$ and $u$ are bounded the first
derivatives of $\si$ are Lipschitzian.
Thus we have
\begin{eqnarray*}
  I_{41} &\le &
    \int_s^t\frac{(t-r)^{\alpha+\beta-1}}{(r-s)^{\alpha}}|\eta^\e(r)|dr\\
&&+     \int_s^t\frac{(t-r)^{\alpha+\beta-1}}{(r-s)^{\alpha}}
\left[ |y^\e(r)| +  |\e \bar u(r)| \right] \left[ | y(r)|+
\bar u(r)\right]dr\,.
\end{eqnarray*}
Since $\sup_{0\le t\le T} |y^\vare(t)|\le C\vare$, and $ \bar u(r)$ and $y(r)$ are bounded,
we see
\begin{equation}
I_{41} \le C  (t-s)^\be \left[ \|\eta^\e\|_{s, t, \infty}+  \e\right] \,.\label{e.4.19}
\end{equation}
In a similar way we have
\begin{eqnarray*}
I_{42} &\le &   \int_s^t\int_s^r\frac{(t-r)^{\alpha+\beta-1}}{(r-\tau)^{\alpha+1}}
|\eta^\e(r)-\eta^\e(\tau)|d\tau dr\\
&&+ \|g\|_\beta\int_s^t\int_s^r\frac{(t-r)^{\alpha+\beta-1}}{(r-\tau)^{\alpha+1}}
\left((r-\tau)^\gamma+|x^*(r)-x^*(\tau)| +\right.\\
&&\left.+|x^\e(r)-x^\e(\tau)| +|u^*(r)-u^*(\tau)| +|u^\e(r)-u^\e(\tau)| \right)
\left| \eta^\e(\tau)\right| d\tau dr\,.
\end{eqnarray*}
By the H\"older continuity of $x^*$ and $u^*$ and unform H\"older
continuity of $x^\e$ and $u^\e$, which are used to assure the integrability,   we
obtain (similar to $\tilde I_2$)
\begin{eqnarray}\label{e.4.20}
I_{42} &\le& (t-s)^{\be+1-\al} \|\eta^\e\|_{1-\al, s, t}+(t-s)^{\be+\nu} \|\eta^\e\|_{s, t, \infty}\,.
\end{eqnarray}
$I_{43}$ and $I_{44}$ are more complex and can be dealt with  in the same way. We shall consider
$I_{43}$.  First we have
\begin{eqnarray}
I_{43}&\le & I_{431}+I_{432}\,,\label{e.4.21}
\end{eqnarray}
where
\begin{eqnarray*}
I_{431}&:=& \int_s^t\int_s^r\frac{(t-r)^{  \al+\be -1 }}
 {(r-\tau)^{\alpha+1}} \int_0^1\left|\sigma_x(r,x^*(r)+\lambda y^\e(r), u^\e(r))-\sigma^*_x(r) \right| d\lambda\\
&&\qquad  \cdot \left|y^\e(r)-y^\e(\tau)\right|
d\tau dr \\
I_{432}&:=& \int_s^t\int_s^r\frac{(t-r)^{  \al+\be -1 }}
 {(r-\tau)^{\alpha+1}}  \int_0^1\Big| \sigma_x(r,x^*(r)+\lambda y^\e(r),u^\e(r))-\sigma^*_x(r)  \\
&&\quad
-\sigma_x(\tau,x^*(\tau)+\lambda y^\e(\tau),u^\e(\tau))+\sigma^*_x(\tau))
\Big|  d\lambda \cdot \left|  y^\e(\tau)\right|d\tau dr\,.
\end{eqnarray*}
By the H\"older continuity of $y^\e$ and Lipschitz continuity of the first
derivatives of $\si$, we have
\begin{eqnarray}
I_{431}&\le & \int_s^t\int_s^r\frac{(t-r)^{  \al+\be -1 }}
 {(r-\tau)^{\alpha+1}} (r-\tau)^{1-\al} \left[ |y^\e(r)|+\e |\bar u(r)|\right]
 dr\le C(t-s)^{\be+1-\al} \e\,.\nonumber\\
 \label{e.4.22}
\end{eqnarray}
To deal with $I_{432}$ we denote
\[
x^\e_{\la, \eta}(r)=x^*(r)+\la \eta y^\e(r)\,,\quad
u^\e_{ \eta}(r)=u^*(r)+  \eta \bar u(r)\,.
\]
Then
\begin{eqnarray*}
&& \sigma_x(r,x^*(r)+\lambda y^\e(r),u^*(r)
+ \e \bar u(r)) -\sigma^*_x(r)\\
&=& \int_0^1 \left[  \la \si_{xx}(r, x^\e_{\la, \eta }(r),u^\e_{\eta }(r)) y^\e(r)
 +  \e
\si_{xu}(r, x^\e_{\la, \eta }(r),u^\e_{\eta }(r))\bar u(r)  \right] d\eta\,.
\end{eqnarray*}
Hence,
\begin{eqnarray*}
I_{432}&\le & \int_s^t\int_s^r\frac{(t-r)^{  \al+\be -1 }}
 {(r-\tau)^{\alpha+1}}  \int_0^1 \int_0^1\bigg[ \left|  \si_{xx}
 (r, x^\e_{\la, \eta }(r),u^\e_{\eta }(r)) y^\e(r)
 -\si_{xx}(\tau , x^\e_{\la, \eta }(\tau ),u^\e_{ \eta }(\tau ))y^\e(\tau)  \right|
   \\
&&\quad
+\e \left|
\si_{xu}(r, x^\e_{\la, \eta }(r),u^\e_{ \eta }(r)) \bar u(r) -
\si_{xu}(\tau , x^\e_{\la, \eta }(\tau ),u^\e_{ \eta }(\tau ))  \bar u(\tau)
 \right| \bigg]  d\lambda d\eta \cdot \left|  y^\e(\tau)\right|d\tau dr\\
\end{eqnarray*}
By the boundedness and the H\"older continuity of the second derivatives of $\si$ we
have
\begin{eqnarray*}
I_{432}&\le & C \int_s^t\int_s^r\frac{(t-r)^{  \al+\be -1 }}
 {(r-\tau)^{\alpha+1}}
|y^\e(r)-y^\e(\tau)| \left|  y^\e(\tau)\right|   dr\\
&&\quad+C\e \int_s^t\int_s^r\frac{(t-r)^{  \al+\be -1 }}
 {(r-\tau)^{\alpha+1}}
|\bar u(r)-\bar u (\tau)| \left|  y^\e(\tau)\right|   dr\\
&&\quad
+ C \int_s^t\int_s^r\frac{(t-r)^{  \al+\be -1 }}
 {(r-\tau)^{\alpha+1}}
 \bigg[ (r-\tau)^\gamma+|x^*(r)-x^*(\tau)|
 +|y^\e(r)-y^\e(\tau)| \\
 &&\quad + \e |u^*(r)-u^*(\tau)| +|u^\e(r)-u^\e(\tau)| \bigg]
  \cdot \bigg[   \left|  y^\e
  (\tau)\right|  +\vare
  \left|  \bar u(\tau)\right| \bigg] \left| y^\e(\tau)\right|  d\tau dr\,.
\end{eqnarray*}
By equation \eref{e.4.14}, we have
\begin{equation}
I_{432} \le   C(t-s)^{\be } \e\,.\label{e.4.23}
\end{equation}
Combination of \eref{e.4.21}-\eref{e.4.23} yields
\begin{equation}
I_{43} \le   C(t-s)^{\be } \e\,.\label{e.4.24}
\end{equation}
In similar way, we have
\begin{equation}
I_{44} \le   C(t-s)^{\be } \e\,.\label{e.4.25}
\end{equation}
The inequalities \eref{e.4.18}, \eref{e.4.19}, \eref{e.4.20}, \eref{e.4.24}, and \eref{e.4.25}
yield
\begin{equation}
I_{4} \le   C\|g\|_\be \left[ (t-s)^{\be+1-\al}  \|\eta^\e \|_{1-\al, s,t}
+(t-s)^{\be}  \|\eta^\e \|_{s, t, \infty} +\e (t-s)^{\be }\right] \,.\label{e.4.26}
\end{equation}

\medskip
\noindent\textbf{Step 6}.\quad
From \eref{e.4.16}, \eref{e.4.17} and \eref{e.4.26}, we have
\begin{equation*}
 |\eta^\e(t)-\eta^\e(s)| \le   C\|g\|_\be \left[ (t-s)^{\be+1-\al}  \|\eta^\e \|_{1-\al, s,t}
+(t-s)^{\be}  \|\eta^\e \|_{s, t, \infty} +\e(t-s)^{\be } \right] \,.
\end{equation*}
This implies
\begin{equation*}
 \|\eta^\e \|_{1-\al, s,t}   \le   C\|g\|_\be \left[ (t-s)^{\be }  \|\eta^\e \|_{1-\al, s,t}
+(t-s)^{\be+\al -1}  \|\eta^\e \|_{s, t, \infty} +\e(t-s)^{\be+\al-1 } \right] \,.
\end{equation*}

Now we can follow the same argument as in Step  3  to complete the proof of the theorem.

{\bf{Step 7.}} The $L^p$ convergence is from the Fernique Theorem.
\end{proof}

As we mentioned earlier Lemma \ref{l.4.3} implies Proposition
\ref{l.3.1} easily. This completes the proof for Proposition
\ref{l.3.1}. \rule{0.5em}{0.5em} \par

To derive the maximum principle, we need to obtain an explicit solution to the equation
\eref{5.3.6}.   Let us   consider   the following   linear matrix-valued
stochastic differential equations.
\begin{eqnarray}
\left\{\begin{array}{rcl}
d\Phi(t)&=&b_x^*(t)\Phi(t)dt+\sum\limits_{j=1}^m\sigma_x^{j, *}(t)\Phi(t)d^\circ B^H_j(t),\\
\Phi(0)&=&I,
\end{array}\right. \label{y1}
\end{eqnarray}
From the basic stochastic calculus for fractional  Brownian motions (see e.g.\cite{HZ}), it is clear  that
this equation has a unique solution, denoted by $\Phi(t)$.
It is easy  to verify that   $\Phi^{-1}(t) $  exists and is
the unique solution of the following stochastic differential equations:
\begin{eqnarray}\left\{\begin{array}{rcl}
d\Phi^{-1}(t)&=&- \Phi^{-1} (t) b_x^*(t)
dt-\sum\limits_{j=1}^m\Phi^{-1}
(t)\sigma_x^{j, *}(t)d^\circ B^H_j(t),\\
\Phi^{-1}(0)&=&I\,.
\end{array}\right.\label{y2}
\end{eqnarray}
Again from   the  It\^o's formula  we can obtain the solution of  the equation   \eref{5.3.6}
by using
 $\Phi(t)$ and $ \Phi^{-1}(t)$.
\begin{lemma} Let $\Phi(t)$ and $\Phi^{-1}(t)$ be defined by \eref{y1} and \eref{y2}. Then the
solution to \eref{5.3.6} is given by
\begin{eqnarray}
y(t)&=&\Phi(t)\int_0^t\Phi^{-1}(s) b_u^*(s)\bar u(s)ds\nonumber\\
&&\qquad
+\sum\limits_{j=1}^m\Phi(t)\int_0^t\Phi^{-1}(s)\sigma_u^{j, *}(s)\bar u(s)dB^H_j(s)\,,   \label{y3}
\end{eqnarray}
\end{lemma}

\section{Maximum principle for stochastic control of system
driven by fractional Brownian motion}
\setcounter{equation}{0}


Recall that we defined the space of admissible controls in Section 4,
\begin{eqnarray}
U[0,T]
&:=&\Bigg\{u| u: [0,T]\times \Omega \to \RR^d, u \hbox{ is } \cG\hbox{-adapted, } u\in C^\mu[0, T]\  \hbox{for some}\ \mu>1-H\,,
\nonumber\\
&&\hbox{there exist constant $C>0$,   $c>0$,  and $\be<H$}  \nonumber\\
&&\quad
\hbox{such that} \ \  \|u\|_{\mu, 0, T}\le C
e^{c \sum_{j=1}^m \|B_j^H\|_{\be, 0, T}}\nonumber\\
&&\quad  \hbox{and}\ \ \int_0^T\int_0^T \EE \left(|D_s^H u(t)|^2\right) dsdt<\infty
\Bigg\}\,.  \label{e.5.1}
\end{eqnarray}
It is clear that $U[0, T]$ is a linear space.

Let   $b: [0,T]\times \RR^n\times \RR^d \to \RR^n$,
$\sigma:[0,T]\times \RR^n \times \RR^d   \to \RR^{n\times m}$,
be   some given continuous  functions
satisfying the assumptions (H4) and (H5) given in Section 4.
Consider the following controlled system of stochastic differential equations driven
by fractional Brownian motions:
\begin{eqnarray}\left\{\begin{array}{rcl}
dx(t)&=&b(t,x(t),u(t))dt+\sum\limits_{j=1}^m\sigma^j(t,x(t),u(t))d^\circ B_j^H(t),\\
x(0)&=&x_0\,. \end{array}\right.\label{6.3.2}
\end{eqnarray}
Here the integral with respect to fractional Brownian motion is the Stratonovich integral.
Some properties of this controlled system are given in Section 4.

Let $l: [0,T]\times \RR^n\times \RR^d \ \to \RR $
and $h: \RR^n\to \RR$ be  some given functions
satisfying the following conditions.

\medskip
\noindent{({\bf H3}) $l$ and $h$ are  continuously differentiable with bounded derivatives.
\medskip

The cost functional we use in this paper  is given by
\begin{eqnarray}
J(u(\cdot))=\EE \left\{\int_0^Tl(t,x(t),u(t))dt+h(x(T))\right\}.\label{e.5.3a}
\end{eqnarray}

Assume $\gamma>1-H $. Let  $\alpha\in (1-H,\alpha_0)$ and $\rho\ge 1/\alpha$, where
$$
\alpha_0=\min\left\{\frac12,\gamma \right\}.
$$
From the conditions (H4)-(H5) the controlled stochastic differential equation (\ref{6.3.2}) has a unique solution  (see e.g.  \cite{HN1}, \cite{NR}). Moreover for $P$-almost all $\omega\in \Omega$, $X(\omega,\cdot)\in C^{1-\alpha}(0,T,\cR^d)$. So assume that $|x(r)-x(\tau)|\le c_0|r-\tau|^{1-\alpha}$ in probability.
The solution $x_t^u$ of above equation depends on $u$.  But
to simplify the notation we often omit its explicit dependence on $u$ and write
$x_t=x_t^u$.

Now our optimal control problem can be stated as to minimize the cost
functional over $U[0,T]$.  That is to find  optimal control $u^*(\cdot)\in U[0, T]$
such that
\begin{equation}
J(u^*(\cdot))=\inf\limits_{u(\cdot)\in U[0,T]}J(u(\cdot))\,.\label{e.5.4a}
\end{equation}
Let $u^*(t)$ be an optimal control and
 $x^*(t)$ be the corresponding solution of equation (\ref{6.3.2}).
$(x^*(\cdot),u^*(\cdot))$  is called an  optimal pair.

We will find a necessary condition that  the optimal control $u^*(\cdot)$ must
satisfy, which is also called the maximum principle.

\subsection{Stochastic control with partial information}

Since $u^*(\cdot)$ is an optimal  control and since the space
$U[0, T]$   of admissible  controls is linear, we have  that $u^*(\cdot)$
is a critical point of nonlinear functional $J(  u(\cdot))$, $u\in U[0, T]$.
  This means that
\begin{eqnarray}
  \left.\frac d{d\e}J(u^*(\cdot)+\e \bar u(\cdot))\right |_{\e=0}=0\,.
  \label{critical}
  \end{eqnarray}
We will follow the same idea as in Section  3 for the Brownian motion
case to deduce a necessary condition from the above equation
\eref{critical}. As in Section 4, we denote $\displaystyle
y(r)=\lim_{\e\rightarrow 0} \frac{ x^\e(r)-x^*(r)}{ \e}$.   By Lemma \eref{l.4.3} and Dominated Convergence Theorem, we have
\begin{eqnarray*}
&~&\left.\frac d{d\e}J(u^*(\cdot)+\e \bar u(\cdot))\right |_{\e=0}
=\EE \int_0^T\left({l_x^*}^T (t)y(t)+ {l_u^*}^T(t)\bar u(t)\right)dt
+\EE \left[ h_x^T(x^*(T))y(T)\right]
 \,.
\end{eqnarray*}
Substituting \eref{e-y} into \eref{critical} we obtain
\begin{eqnarray*}
&~&\left.\frac d{d\e}J(u^*(\cdot)+\e \bar u(\cdot))\right |_{\e=0}
= \EE \int_0^T\left({l_x^*}^T (t)\Phi(t)\int_0^t\Phi^{-1}(s){b_u^*} (s)
\bar u(s)ds+{l_u^*}^T (t)\bar u(t)\right)dt\\
&~&+\EE \left\{h_x^T (x^*(T))\Phi(T)\int_0^T\Phi^{-1}(s){b_u^*} (s)\bar u(s)ds\right\}\\
&~&+\sum\limits_{j=1}^m\EE \int_0^T\left({l_x^*}^T(t)\Phi(t)
\int_0^t\Phi^{-1}(s)\sigma_u^{j, *}(s)
\bar u(s)d^\circ B_j^H(s)\right)dt\\
&~&+\sum\limits_{j=1}^m\EE \left\{h_x^T(x^*(T))\Phi(T)
\int_0^T\Phi^{-1}(s)\sigma_u^{j, *}(s)\bar u(s)d^\circ B_j^H(s)\right\}\\
&=&\EE \int_0^T\left(\int_s^T\left({l_x^*}^T(t)\Phi(t)\right)dt
\Phi^{-1}(s)b_u^*(s)\bar u(s)\right)ds+\EE \int_0^T {l_u^*}^T(s)\bar u(s)ds\\
&~&+\EE \int_0^T h^T_x(x^*(T))\Phi(T)\Phi^{-1}(s)b_u^*(s)\bar u(s)ds\\
&~&+\sum\limits_{j=1}^m\int_0^T\left(\EE {l_x^*}^T (t)\Phi(t)
\int_0^t\left(\Phi^{-1}(s)\sigma_u^{j, *}(s)\bar u(s)d^\circ B_j^H(s)\right)ds\right)dt\\
&~&+\sum\limits_{j=1}^m\EE \left\{h_x^T(x^*(T))\Phi(T)
\int_0^T\Phi^{-1}(s)\sigma_u^{j, *}(s)\bar u(s)d^\circ B_j^H(s)\right\}.
\end{eqnarray*}
%
The expectation of the above last two terms can be computed by the formula
\eref{expectation}.  Thus,  we have
\begin{eqnarray*}
&~&\left.\frac d{d\e}J(u^*(\cdot)+\e \bar u(\cdot))\right |_{\e=0}\\
&=&\EE \int_0^T \left( \int_s^T  {l_x^*}^\top(t)\Phi(t)dt\right
)\Phi^{-1}(s)b_u^*(s)\bar u(s)ds+\EE \int_0^T  {l_u^*}^\top(s)\bar u(s)ds\\
&~&+\EE \int_0^T h_x^\top(x^*(T))\Phi(T)\Phi^{-1}(s)b_u^*(s)\bar u(s)ds\\
&~&\qquad +\sum\limits_{j=1}^m\EE \int_0^T \left(\int_s^T  \DD_s^j
\left\{{l_x^*}^\top(t)\Phi(t)\Phi^{-1}(s)\sigma_u^{j,
*}(s)\right\}\bar u(s)dt\right)ds\\ &~&+\sum\limits_{j=1}^m\EE
\int_0^T \int_s^T {l_x^*}^\top(t)\Phi(t)\Phi^{-1}(s)\sigma_u^{j,
*}(s)\DD_s^j  \left(\bar u(s)\right)dtds
\\
&~&\qquad +\sum\limits_{j=1}^m\EE \int_0^T \DD_s^j
 \left(h_x^\top(x^*(T))\Phi(T)\Phi^{-1}(s)\sigma_u^{j,*}(s)\right)\bar
 u(s)ds\\ &~&+\sum\limits_{j=1}^m\EE
\int_0^T h_x^\top(x^*(T))\Phi(T)\Phi^{-1}(s)\sigma_u^{j,
*}(s)\DD_s^j \left(\bar u(s)\right)ds.
\end{eqnarray*}
Since the equation \eref{critical} holds true for all adapted
process $u$ in $U[0, T]$, we can   choose especially $\bar
u(s)=\mathbf{1}_{[a,b]}\tilde u$, where $0\le a\le b\le T$ and $
\tilde{u}$ which is ${\cal G}_a$ measurable. Then from
\eref{critical}  and the above computation it follows
\begin{eqnarray*}
&&\EE \int_a^b  \left( \int_s^T  {l_x^*}^\top(t)\Phi(t)dt\right
)\Phi^{-1}(s)b_u^*(s)\bar u      ds+\EE \int_a^b   {l_u^*}^\top(s)\bar u      ds\\
&~&+\EE \int_a^b  h_x^\top(x^*(T))\Phi(T)\Phi^{-1}(s)b_u^*(s)\bar u      ds\\
&~&\qquad +\sum\limits_{j=1}^m\EE \int_a^b  \left(\int_s^T
\DD_s^j \left\{{l_x^*}^\top(t)\Phi(t)\Phi^{-1}(s)\sigma_u^{j,
*}(s)\right\}\bar u      dt\right)ds\\ &~&+\sum\limits_{j=1}^m\EE
\int_a^b  \int_s^T {l_x^*}^\top(t)\Phi(t)\Phi^{-1}(s)\sigma_u^{j,
*}(s)\DD_s^j  \left(\bar u      \right)dtds
\\
&~&\qquad +\sum\limits_{j=1}^m\EE \int_a^b  \DD_s^j
 \left(h_x^\top(x^*(T))\Phi(T)\Phi^{-1}(s)\sigma_u^{j,*}(s)\right)\bar
 u ds\\ &~&+\sum\limits_{j=1}^m\EE
\int_a^b  h_x^\top(x^*(T))\Phi(T)\Phi^{-1}(s)\sigma_u^{j,
*}(s)\DD_s^j \left(\bar u      \right)ds=0.
\end{eqnarray*}
We can use the formula \eref{ito-expectation} to compute the above
two terms involving $\DD_s^j  \bar u$. we have then
\begin{eqnarray*}
&& \EE \Bigg[ \Bigg\{ \int_a^b  \left( \int_s^T
{l_x^*}^\top(t)\Phi(t)dt\right
)\Phi^{-1}(s)b_u^*(s)      ds+\int_a^b   {l_u^*}^\top(s)       ds \\
&~&+\int_a^b  h_x^\top(x^*(T))\Phi(T)\Phi^{-1}(s)b_u^*(s)      ds\\
&~&\qquad +\sum\limits_{j=1}^m\int_a^b  \left(\int_s^T  \DD_s^j
\left\{{l_x^*}^\top(t)\Phi(t)\Phi^{-1}(s)\sigma_u^{j,
*}(s)\right\}      dt\right)ds\\ &~&+\sum\limits_{j=1}^m
\int_a^b \left[
\int_s^T {l_x^*}^\top(t)\Phi(t)\Phi^{-1}(s)\sigma_u^{j,
*}(s) dt \right]  dB_j^H(s)
\\
&~&\qquad +\sum\limits_{j=1}^m\int_a^b  \DD_s^j
 \left(h_x^\top(x^*(T))\Phi(T)\Phi^{-1}(s)\sigma_u^{j,*}(s)\right)\bar
 u ds\\ &~&+\sum\limits_{j=1}^m
\int_a^b  h_x^\top(x^*(T))\Phi(T)\Phi^{-1}(s)\sigma_u^{j,
*}(s) dB_j^H(s) \Bigg\}\bar u  \Bigg]  =0.
\end{eqnarray*}
Since the $\cG  _a$ measurable $\bar u$ is arbitrary,   we have
\begin{theorem}  Let $u^*$ be the optimal admissible control.  Then $u^*$ satisfies
the following
\begin{eqnarray}
&& \EE \Bigg[ \Bigg\{ \int_a^b  \left[ \int_s^T
{l_x^*}^\top(t)\Phi(t)dt  \Phi^{-1}(s)b_u^*(s)  +  {l_u^*}^\top(s)
   +
  h_x^\top(x^*(T))\Phi(T)\Phi^{-1}(s)b_u^*(s)  \right]     ds\nonumber\\
&~&  +\sum\limits_{j=1}^m\int_a^b  \bigg[ \int_s^T  \DD_s^j
\left\{{l_x^*}^\top(t)\Phi(t)\Phi^{-1}(s)\sigma_u^{j,
*}(s)\right\}      dt \nonumber\\
&&\quad +  \DD_s^j
 \left\{ h_x^\top(x^*(T))\Phi(T)\Phi^{-1}(s)\sigma_u^{j,*}(s)\right\}  \bigg]  ds
+\sum\limits_{j=1}^m \int_a^b \bigg[
\int_s^T {l_x^*}^\top(t)\Phi(t)\Phi^{-1}(s)\sigma_u^{j,
*}(s) dt \nonumber \\
&&\quad  +
  h_x^\top(x^*(T))\Phi(T)\Phi^{-1}(s)\sigma_u^{j,
*}(s)\bigg]  dB_j^H(s) \Bigg\} \bigg|\cG  _a  \Bigg]  =0\,,\quad \forall \ \ 0\le a\le b\le T\,.  \label{partial}
\end{eqnarray}
\end{theorem}
\begin{remark} Using pathwise integral, we can write the above equation as
\begin{eqnarray}
&& \EE \Bigg[ \Bigg\{ \int_a^b  \left[ \int_s^T
{l_x^*}^\top(t)\Phi(t)dt  \Phi^{-1}(s)b_u^*(s)  +  {l_u^*}^\top(s)
   +
  h_x^\top(x^*(T))\Phi(T)\Phi^{-1}(s)b_u^*(s)  \right]     ds\nonumber\\
&~&   \sum\limits_{j=1}^m \int_a^b \bigg[
\int_s^T {l_x^*}^\top(t)\Phi(t)\Phi^{-1}(s)\sigma_u^{j,
*}(s) dt \nonumber \\
&&\quad  +
  h_x^\top(x^*(T))\Phi(T)\Phi^{-1}(s)\sigma_u^{j,
*}(s)\bigg]  d^\circ B_j^H(s) \Bigg\} \bigg|\cG  _a  \Bigg]  =0\,,\quad \forall \ \ 0\le a\le b\le T\,.
\end{eqnarray}
\end{remark}

\setcounter{equation}{0}
\subsection{Stochastic control with complete information}
Now we assume that the filtration $\cG=\cF$ and note that it is also the filtration generated by  the (background)  defining  Brownian motions $W$.   Namely,
\[
\cF  _t=\si\left(B_1^H(s)\,, \cdots\,, B_m^H(s)\,, \ 0\le s\le t \right)
=\si\left(W_1 (s)\,, \cdots\,, W_m (s)\,, \ 0\le s\le t \right)
\,.
\]
We shall simplify the equation \eref{partial}.   Denote
\begin{eqnarray*}
F(T,s)&=&\left(\int_s^T{l_x^*}^\top (t)\Phi(t)dt\right)\Phi^{-1}(s)b_u^*(s)+{l_u^*}^\top (s)+h_x^\top (x^*(T))\Phi(T)\Phi^{-1}(s)b_u^*(s)\,,\\
G_j (T,s)&=& \left(\int_s^T{l_x^*}^\top (t)\Phi(t)dt\right)\Phi^{-1}(s)\sigma^{j, *}_u(s)+h_x^\top (x^*(T))\Phi(T)\Phi^{-1}(s)\sigma_u^{j, *}(s) \,, \\
\tilde F(T,s)&=&  F(T,s)+\sum_{j=1}^m \mathbb{D}^j_s G_j (T,s)\,.
\end{eqnarray*}
Then  the equation \eref{partial} can be written as
\begin{equation}
\EE \left\{ \int_a^b \tilde F(T, s) ds+\sum_{j=1}^m \int_a^b G_j (T, s)dB_j^H(s)\Big|
\cF  _a\right\}=0\,.\label{e.6.1}
\end{equation}
We shall simplify the above equation \eref{e.6.1} in the case of complete information.
To this end we need some lemmas.

\begin{lemma}\label{l.6.1}  Let $1/2<H<1$, and $\e>0$.
Denote
\[
\rho(\e)=\int_0^a\int_0^a(t-s)^{2H-2}s^{\frac12-H}(a-s+\e)^{-H-\frac12}t^{\frac12-H}(a-t+\e)^{-H-\frac12}dsdt.
\]
For any $\delta\in(0,2-2H),$  there exist $C_\de >0$,  depending on $H, a$ and $\de$ but independent of $\e$,   such that
\begin{equation}
\rho(\e)\ge C_\de   \e^{1-2H-\delta}\,.
\end{equation}
\end{lemma}
\begin{proof} Without loss of generality, we prove the case when $a=1.$\\
We choose some $\delta\in(0,2-2H)$, and we have
 $|t-s|^{2H-2}(1-s+\e)^{\delta}=|t-s|^{2H-2+\delta}(\frac{1-s+\e}{t-s})^{\delta}\ge1, $  when $0\le s\le t\le 1$.
\begin{align*}
&\int_0^1\int_0^1|t-s|^{2H-2}s^{\frac12-H}t^{\frac12-H}(1-s+\e)^{-H-\frac12}(1-t+\e)^{-H-\frac12}dsdt\\
=&2 \int_0^1\int_0^t|t-s|^{2H-2}s^{\frac12-H}t^{\frac12-H}(1-s+\e)^{-H-\frac12}(1-t+\e)^{-H-\frac12}dsdt\\
\ge &2 \int_0^1\int_0^t|t-s|^{2H-2}(1-s+\e)^{-H-\frac12}(1-t+\e)^{-H-\frac12}dsdt\\
\ge &2 \int_0^1\int_0^t(1-s+\e)^{-H-\frac12-\delta}(1-t+\e)^{-H-\frac12}dsdt\\
=&C \int_0^1 \left[(1-t+\e)^{\frac12-H-\delta}-(1+\e)^{\frac12-H-\delta}\right](1-t+\e)^{-H-\frac12}dt\\
=&C\left(\frac1{2H+\delta-1} [\e^{1-2H-\delta}-(1+\e)^{1-2H-\delta}]-\frac1{H-\frac12}[\e^{-H+\frac12}-(1+\e)^{-H+\frac12}](1+\e)^{\frac12-H-\delta}\right)\\
\sim& \e^{1-2H-\delta} \text{ as } \e \to 0.
\end{align*}
\end{proof}

\begin{lemma}\label{l.6.2}
Let $X_t$  be  a Gaussion random variable with 0 mean and variance $f^2(t)$.   If   $\lim_{t\to 0}f(t)=\infty$,  then
$\lim_{t\to 0}|X_t|=\infty$ in probability.
\end{lemma}
\begin{proof}  We have
\begin{align*}
E(e^{-|X_t|})&=\frac2{\sqrt{2\pi}}\int_0^\infty e^{-f(t)x}e^{-\frac {x^2}{2}}dx\\
&\le \frac2{\sqrt{2\pi}} \int_0^\infty e^{-f(t)x}dx = \frac2{\sqrt{2\pi}} \frac1{f(t)}\,.
\end{align*}
This implies that  $\lim_{t\to0}E(e^{-|X_t|})=0$
and hence  $e^{-|X_t|}$ goes to $0$ in $L^1$.  Consequently,  we have  $\lim_{t\to 0}|X_t|=\infty$ in probability as $t\to 0$.
\end{proof}

\begin{lemma}\label{l.6.3}  Let $g$ be continuous on $ [0, T]$  and let $f_j\,, j=1, 2, \cdots, m$ be
H\"older continuous on $[0, T]$ of   order $\rho$ with $\rho>1-H$.   Assume that the Malliavin derivative
$\DD_s^j f_j(t)$ is continuous in $s\in [0, T]$ for all  $t\in [0, T]$ and assume that
$\displaystyle \sup_{0\le s\le T} \DD_s^j \EE \left[ f_j(t)|\cF  _t\right] <\infty$ almost surely for all  $t\in [0, T]$.
If
\begin{equation}
\EE\left\{ \int_a^b g(s) ds + \sum_{j=1}^m \int_a^b f_j (s)d^\circ B_j^H(s) \Big|\mathcal F_a \right\}
=0 \,, \quad \forall \ \ 0\le a\le b\le T\,, \label{e.6.3}
\end{equation}
 then  for all $j=1, 2, \cdots, m$,
\begin{equation}
\EE\left[   f_j(a)  |  \mathcal F_a \right]=0 \,, \quad \forall \ \ 0< a \le T\,.\label{e.6.4}
\end{equation}
\end{lemma}
\begin{proof}

\noindent\textbf{Step 1}.\quad  Let $\e>0$ and $0<a<a+\e<T$.    The equation \eref{e.6.3}  can be written as
\begin{eqnarray}
&&\EE\left\{ \int_a^{a+\e}  g(s) ds + \sum_{j=1}^m \int_a^{a+\e}  \left[  f_j (s)-f_j(a)\right]
d^\circ B_j^H(s) \Big|\mathcal F_a \right\}  \nonumber\\
&&\qquad
+\sum_{j=1}^m  \EE \left\{ f_j(a) \left[ B_j^H(a+\e)-B_j^H(a)\right]  \Big|\mathcal F_a \right\}
=0 \, \label{e.6.5}
\end{eqnarray}
Let   $0<\be<H$ such that $\be+\rho> 1$.   Then from \eref{e.2.5},  it follows  that
\begin{eqnarray*}
\left| \int_a^{a+\e }  \left[  f_j (s)-f_j(a)\right]
d^\circ B_j^H(s)\right|
&\le & \|B_j^H\|_{\be, 0, T} \e^{\be +\rho} \,.
\end{eqnarray*}
Dividing \eref{e.6.5} by $\e^{H'}$ with  $0<H'<1$, then
we have
\begin{equation}
\lim_{\e\rightarrow  0}
\frac{1}{\e^{H'} } \sum_{j=1}^m  \EE \left\{ f_j(a) \left[ B_j^H(a+\e)-B_j^H(a)\right]  \Big|\mathcal F_a \right\}   =0\,.
\label{e.6.6}
\end{equation}
It is maybe possible to compute the above expectation in an easy way.  But instead of developing a formula
for above conditional expectation we shall use the results from \cite{Hu1}.  First, we have
\[
\int_a^{a+\e} f_j(a) dB_j^H(s)=  f_j(a) \left[ B_j^H(a+\e)-B_j^H(a)\right] -\int_a^{a+\e} \DD_s^j f_j(a) ds
\]
By the  continuity of $\DD_s f_j(a)$,  \eref{e.6.6} implies
\begin{equation}
\lim_{\e\rightarrow  0}
\frac{1}{\e^{H'} } \sum_{j=1}^m  \EE \left\{ \int_a^{a+\e} f_j(a) dB_j^H(s)
 \Big|\mathcal F_a \right\}   =0\,, \quad \forall a\in [0, T]\,.
\label{e.6.7}
\end{equation}

\medskip
\noindent\textbf{Step 2}.\quad Let us recall  some  notations in \cite{Hu1}  (see Equations  (5.21), (5.34),   and (9.22) of \cite{Hu1}).
\begin{eqnarray*}
&&{\bf\Gamma}_{H,T}^*f(t)=(H-\frac12)\kappa_Ht^{\frac12-H}\int_t^T\xi ^{H-\frac12}(\xi -t)^{H-\frac32}f(\xi )d\xi \,,\\
&&{\mathbb B}_{H,\tau}^*g(t)=-\frac{2H\kappa_1}{\kappa_H}t^{\frac12-H}\frac{d}{dt}\int_t^\tau(\eta -t)^{\frac12-H}\eta ^{H-\frac12}g(\eta )d\eta \,,
\end{eqnarray*}
and
\[
{\mathbb P}_{H,\tau}(t)f(t)={\mathbb B}^*_{H,\tau}{\bf \Gamma}_{H,T}^*f(t)\,.
\]
Then from the equation (9.21) of \cite{Hu1}, we have
\begin{eqnarray}
&~&\EE\left[\left.\int_a^{a+\e}f_j(a)dB^H_j(s) \right|{\cal F}_a\right]\nonumber  \\
&=&\int_0^a{\mathbb P}_{H,a}(s)\EE\left[\left.{\bf 1}_{[a,a+\e]}(s)f_j(a) \right|{\cal F}_a\right]dB^H_j(s) \nonumber  \\
&=&\int_0^a\left\{C
s^{\frac12-H}\frac{d}{ds}\int_s^a\left[(\eta -s)^{\frac12-H}\eta ^{H-\frac12}(H-\frac12)\kappa_H\eta ^{\frac12-H}\right.\right.
\nonumber  \\
&&\cdot \left.\left.\int_\eta ^T\left(\xi ^{H-\frac12}(\xi -\eta )^{H-\frac32}\EE\left[{\bf 1}_{[a,a+\e]}(\xi )f_j(a )|{\cal F}_a\right]\right)d\xi  \right]d\eta  \right\}dB^H_j(s) \nonumber \\
&=&C \int_0^a \zeta(a, \e, s) \EE\left[f_j(a )|{\cal F}_a\right] dB^H_j(s) \nonumber \\
&=&C\EE\left[f_j(a )|{\cal F}_a\right]  \int_0^a \zeta(a, \e, s)   dB^H_j(s) -C\int_0^a \zeta(a, \e, s)  \DD_s^j \left(\EE(f_j(a)|\cF  _a]
\right) ds \,,  \nonumber \\
\label{e.6.8}
\end{eqnarray}
where
\[
\zeta(a, \e, s)  =
s^{\frac12-H}\frac{d}{ds}\int_s^a(\eta -s)^{\frac12-H}\\
  \left[\int_a^{a+\e}  \xi ^{H-\frac12}(\xi -\eta )^{H-\frac32}
d\xi  \right]d\eta  \,.
\]

\medskip
\noindent\textbf{Step 3}.\quad
This function $\zeta$ can be calculated as follows.
\begin{eqnarray*}
\zeta(a, \e, s)
&=& s^{\frac12-H} \frac{d}{ds}\int_a^{a+\e}\int_s^a  (\eta -s)^{\frac12-H}  (\xi -\eta )^{H-\frac32}  \xi ^{H-\frac12} d\eta  d\xi \\
&=& s^{\frac12-H} \frac{d}{ds}\int_a^{a+\e}\left( \int_s^\xi   (\eta -s)^{\frac12-H}  (\xi -\eta )^{H-\frac32}\   d\eta \right.\\
&&~~~~\left.-\int_a^\xi   (\eta -s)^{\frac12-H}  (\xi -\eta )^{H-\frac32}\  d\eta \right) \xi ^{H-\frac12}d\xi \\
&=&  -s^{\frac12-H}\frac{d}{ds}\int_a^{a+\e}\int_a^\xi   (\eta -s)^{\frac12-H}  (\xi -\eta )^{H-\frac32}  d\eta  \xi ^{H-\frac12}d\xi \\
&=&  s^{\frac12-H}\int_a^{a+\e}\int_a^\xi  (\frac 12 -H) (\eta -s)^{-\frac12-H}  (\xi -\eta )^{H-\frac32}  d\eta \xi ^{H-\frac12} d\xi \\
&=& \e^{H+\frac12}s^{\frac12-H} \int_0^1 \int_0^1 (\eta '\xi '\e+a-s)^{-\frac12-H}(1-\eta ')^{H-\frac32} \xi '^{H-\frac12}\\
&&\qquad (\xi '\e+a)^{H-\frac12} d\eta 'd\xi '\,.
\end{eqnarray*}
 Choose $ \ga \in(0,1)$.
Dividing \eref{e.6.8} by $\e^\ga$,    we see
\[
\lim_{\e\rightarrow 0}   \sum_{j=1}^m  \EE\left[ f_j (a )|{\cal F}_a\right]   \frac{1}{\e^{\ga} }   \int_0^a \zeta(a, \e, s)   dB^H_j(s)=0\,.
\]
But
\begin{align*}
 \frac{1}{\e^{\ga} }   \int_0^a \zeta(a, \e, s)   dB^H_j(s)
&= \e^{H+\frac12-\ga }  \int_0^1\int_0^1 \left(\int_0^a s^{\frac12-H}(\eta \xi \e+a-s)^{-\frac12-H}dB_j(s)^H\right)\\
&~~~~\times(1-\eta )^{H-\frac32} \xi ^{H-\frac12}(\xi \e+a)^{H-\frac12} d\eta  d\xi
\end{align*}
and hence
\begin{align*}
&\EE \left(\frac{1}{\e^{\ga} }   \int_0^a \zeta(a, \e, s)   dB^H_j(s)\right)^2\\
=& C\e^{2H+1-2\ga} \int_{[0,1]^4} \int_0^a\int_0^a |t-s|^{2H-2}s^{\frac12-H}t^{\frac12-H}(\eta_1\xi_1 \e+a-s)^{-\frac12-H}(\eta_2\xi_2 \e+a-t)^{-\frac12-H}dsdt\\
&\quad\quad\qquad\times \prod_{i=1}^2(1-\eta_i )^{H-\frac32} \xi_i ^{H-\frac12}(\xi_i \e+a)^{H-\frac12}d\eta_1  d\xi_1  d\eta_2  d\xi_2 \\
\ge &C\e^{2H+1-2\ga} \int_{[0,1]^4} \int_0^a\int_0^a |t-s|^{2H-2}s^{\frac12-H}t^{\frac12-H}( \e+a-s)^{-\frac12-H}( \e+a-t)^{-\frac12-H}dsdt\\
&\quad\quad\qquad\times \prod_{i=1}^2(1-\eta_i )^{H-\frac32} \xi_i ^{H-\frac12}(\xi_i \e+a)^{H-\frac12}d\eta_1  d\xi_1  d\eta_2  d\xi_2 \\
=& C_\de\e^{2H+1-\ga+1-2H-\de}\\
=& C_\de\e^{2-2\ga-\de}
\end{align*}
We may choose $\ga\in(0,1)$ and $\de\in(0,2-2H)$ such that $2-2\ga-\de<0.$ From Lemma \ref{l.6.2} we see that   $\frac{1}{\e^{\ga} }   \int_0^a \zeta(a, \e, s)   dB^H_j(s) $ converges to $\infty$ in probability as $\e\rightarrow 0$. This implies that $\EE\left[ f_j (a )|{\cal F}_a\right]=0$ a.s..
\end{proof}

Then we have also the following lemma.
\begin{lemma}\label{l.6.4}  Let $g$ be continuous on $ [0, T]$  and let $f_j\,, j=1, 2, \cdots, m$ be
H\"older continuous on $[0, T]$ of   order $\rho$ with $\rho>1-H$.   Assume that the Malliavin derivative
$\DD_s^j f_j(t)$ is continuous in $s\in [0, T]$ for all  $t\in [0, T]$ and assume that
$\displaystyle \sup_{0\le s\le T} \DD_s^j \EE \left[ f_j(t)|\cF  _t\right] <\infty$ almost surely for all  $t\in [0, T]$.
If
\begin{equation}
\EE\left\{ \int_a^b g(s) ds + \sum_{j=1}^m \int_a^b f_j (s)d  B_j^H(s) \Big|\mathcal F_a \right\}
=0 \,, \quad \forall \ \ 0\le a\le b\le T\,, \label{e.6.9}
\end{equation}
 then  for all $j=1, 2, \cdots, m$,
\begin{equation*}
\EE\left[   f_j(a)  |  \mathcal F_a \right]=0  \quad \label{e.6.10}
\end{equation*}
and \[\EE \left[ g(a) |\mathcal F_a\right]=0 \]
 for all $ 0< a \le T\,.$
\end{lemma}
\begin{proof} The first equality is obvious from the previous lemma. Now we prove the second one.
For any $\e>0,$
\begin{align*}
&\EE\left[\left.\int_a^{a+\e}f_j(s)dB^H_j(s) \right|{\cal F}_a\right]\nonumber  \\
=&\int_0^a{\mathbb P}_{H,a}(s)\EE\left[\left.{\bf 1}_{[a,a+\e]}(s)f_j(a) \right|{\cal F}_a\right]dB^H_j(s) \nonumber  \\
&\cdot \left.\left.\int_\eta ^T\left(\xi ^{H-\frac12}(\xi -\eta )^{H-\frac32}\EE\left[{\bf 1}_{[a,a+\e]}(\xi )f_j(\xi )|{\cal F}_a\right]\right)d\xi  \right]d\eta  \right\}dB^H_j(s) \nonumber \\
=&\int_0^a{\mathbb P}_{H,a}(s)\EE\left[\left.{\bf 1}_{[a,a+\e]}(s)f_j(a) \right|{\cal F}_a\right]dB^H_j(s) \nonumber  \\
&\cdot \left.\left.\int_\eta ^T\left(\xi ^{H-\frac12}(\xi -\eta )^{H-\frac32}\EE\left[\EE\left[{\bf 1}_{[a,a+\e]}(\xi )f_j(\xi )|{\cal F}_\xi|{\cal F}_a \right]\right]\right)d\xi  \right]d\eta  \right\}dB^H_j(s) \nonumber \\
=&0,
\end{align*}
then we have 
\[\EE[\int_a^{a+\e} g(s)ds|\mathcal F_a]=0, \forall \e>0,\]
and hence \[\EE[g(a)|\mathcal F_a]=0.\]
\end{proof}

We apply Lemma \ref{l.6.4} to the equation \eref{e.6.1}  and obtain
\begin{equation}
\EE \left[ G_j(T, t) \Big| \cF  _t\right]=0\,,\label{e.6.11}
\end{equation}
and
\begin{equation}
 \EE \left[ \tilde F(T, t) \Big| \cF  _t \right]=0\,,  \label{e.6.14}
\end{equation}
for all $0<t\le T.$
Denote
\begin{eqnarray}
P(t)&=& {\left(\Phi^{\top}\right) }^{ -1}(t) \left[\int_t^T\Phi^\top (s)l_x^*(s)ds+\Phi^\top (T)h_x(x^*(T)) \right]\nonumber\\
p(t)&=&{\left(\Phi^{\top}\right) }^{ -1}(t)\EE ^{{\cal F}_t}\left[\int_t^T\Phi^\top (s)l_x^*(s)ds+\Phi^\top (T)h_x(x^*(T))\right]\,.
 \label{e.6.12}
\end{eqnarray}
Then the equations \eref{e.6.11} and \eref{e.6.14} can be written as
\begin{equation}
\sum_{j=1}^m \si_u^{j, *\top } (t) p(t) =0\,, \quad \forall \ 0\le t\le T\, \label{e.6.13}
\end{equation}

and

\begin{equation}
b_u^{*\top}(t)p(t)
+l_u^*(t)+ \sum\limits_{j=1}^m \EE \left[ \DD_t ^j G_j(T, t)\Big|\cF  _t\right] =0\,,
\quad \forall \  0\le t \le  T\,. \label{e.6.15}
\end{equation}

Now we compute  $\EE \left[ \DD_t ^j G_j(T, t)\Big|\cF  _t\right]$.
Unlike in the classical case, which we have
$\EE \left[  D_t ^j G_j(T, t)\Big|\cF  _t\right]= D_t ^j\EE \left[   G_j(T, t)\Big|\cF  _t\right]$,  now we usually have
$\EE \left[ \DD_t ^j G_j(T, t)\Big|\cF  _t\right]\not= \DD_t ^j \EE \left[   G_j(T, t)\Big|\cF  _t\right]$.  We need to use \eref{e.2.19}.   Let $c_{1, H}$ be a constant as defined in  proposition \ref{p.2.4}  and $\phi_{1, H}(s, t)=c_{1, H}s^{\frac12-H} |t-s|^{2H-2}$. Then
\begin{eqnarray}
\EE \left[ \DD_t ^j G_j(T, t)\Big|\cF  _t\right]
 &=&   \int_0^T\phi_{1, H}(t-s)  \left( \frac{d}{ds} \int_s^T
(r-s)^{\frac12-H}r^{H-\frac12}     \EE \left[  D_r ^j G_j(T, t)\Big|\cF  _t\right] dr\right) ds \nonumber \\
&=&I_1+I_2\,,  \label{e.6.16}
\end{eqnarray}
where
\begin{eqnarray*}
I_1&=&   \int_0^t\phi_{1, H}(t-s)  \left( \frac{d}{ds} \int_s^t
(r-s)^{\frac12-H}r^{H-\frac12}     \EE \left[  D_r ^j G_j(T, t)\Big|\cF  _t\right] dr\right) ds\nonumber  \\
I_2&=&  \int_0^T\phi_{1, H}(t-s)  \left( \frac{d}{ds} \int_{t\vee s} ^T
(r-s)^{\frac12-H}r^{H-\frac12}     \EE \left[  D_r ^j G_j(T, t)\Big|\cF  _t\right] dr\right) ds\,.
\end{eqnarray*}
Using Proposition 1.2.8 of \cite{N}, we have
\begin{eqnarray}
I_1&=&   \int_0^t\phi_{1, H}(t-s)  \left( \frac{d}{ds} \int_s^t
(r-s)^{\frac12-H}r^{H-\frac12}      D_r ^j  \EE \left[  G_j(T, t)\Big|\cF  _t\right] dr\right) ds
\nonumber \\
&=&   \int_0^t\phi_{1, H}(t-s)  \left( \frac{d}{ds} \int_s^t
(r-s)^{\frac12-H}r^{H-\frac12}      D_r ^j   \EE \left[  P (t)\si_u^{j,*}(t)   \Big|\cF  _t\right]   dr\right) ds \nonumber \\
&=&   \int_0^t\phi_{1, H}(t-s)  \left( \frac{d}{ds} \int_s^t
(r-s)^{\frac12-H}r^{H-\frac12}      D_r ^j   \left(p(t)  \si_u^{j, *}(t)\right)
  dr\right)   ds \,.\nonumber  \\\label{e.6.17}
\end{eqnarray}
$I_2$ is computed as follows
 \begin{eqnarray}
 I_2&=&  \int_0^T\phi_{1, H}(t-s)  \left( \frac{d}{ds} \int_{t\vee s} ^T
(r-s)^{\frac12-H}r^{H-\frac12}     \EE \left[  D_r ^j \left(P(t) \si_u^{j, *}(t)\right)
\Big|\cF  _t\right] dr\right) ds\nonumber  \\
&=&   \si_u^{j, *}(t) \int_0^T\phi_{1, H}(t-s)  \left( \frac{d}{ds} \int_{t\vee s} ^T
(r-s)^{\frac12-H}r^{H-\frac12}      \EE \left[  D_r ^j  P(t)
\Big|\cF  _t\right] dr\right) ds \,. \nonumber  \\ \label{e.6.18}
\end{eqnarray}

Now let us discuss $p(t)$.  First we have
\begin{eqnarray*}
p(t)&=&{\left(\Phi^{\top}\right) }^{ -1}(t)\EE ^{{\cal F}_t}\left[\int_0^T\Phi^\top (s)l_x^*(s)ds+\Phi^\top (T)h_x(x^*(T))\right]
\nonumber\\
&&\qquad - {\left(\Phi^{\top}\right) }^{ -1}(t)  \int_0^t\Phi^\top (s)l_x^*(s)ds  \,.
\end{eqnarray*}
Since  $\left(
\EE ^{{\cal F}_t}\left[\int_0^T\Phi^\top (s)l_x^*(s)ds+\Phi^\top (T)h_x(x^*(T))\right]\,, 0\le t\le T \right)$
 is a square integrable martingale with respect to the filtration
 $\cF  _t$ generated  by the standard Brownian motion $W$,  we have
 \[
 \EE ^{{\cal F}_t}\left[\int_0^T\Phi^\top (s)l_x^*(s)ds+\Phi^\top (T)h_x(x^*(T))\right]=-\sum_{j=1}^m \int_0^t \Phi^\top (s)
 q_j(s) dW_j(s)\,,
 \]
 where the introduction of the factor $-\Phi^\top (s)$ is to simplify the equation obtained subsequently.   Therefore,
 \begin{eqnarray*}
p(t)&=&-{\left(\Phi^{\top}\right) }^{ -1}(t) \sum_{j=1}^m \int_0^t \Phi^\top (s) q_j(s) dW_j(s)
- {\left(\Phi^{\top}\right) }^{ -1}(t)  \int_0^t\Phi^\top (s)l_x^*(s)ds  \,.
\end{eqnarray*}
Denote
\[
K_t= -\sum_{j=1}^m \int_0^t \Phi^\top (s) q_j(s) dW_j(s)
\]
and
\[
A_t=  \int_0^t\Phi^\top (s)l_x^*(s)ds\,.
\]
Then by Proposition \ref{p.2.5} and equation \eref{y2},   we have
\begin{eqnarray*}
dp(t)&=&d\Phi^{\top, -1}(t)\cdot K_t+\Phi^{\top, -1}(t)\cdot dK_t-d\Phi^{\top,-1}(t)\cdot A(t)-\Phi^{\top, -1}(t)\cdot dA(t)\\
&=&-b_x^{*\top}(t)p(t)dt-l_x^*(t)dt-\sigma_x^{*\top}(t)p(t)d^\circ B^H(t)- q(t)dW(t),\\
\end{eqnarray*}
It is clear that
\[
p(T)=h_x(x^*(T)),
\]

Therefore, $p(t)$ satisfies the following backward stochastic differential equation
\begin{equation}
\begin{cases}
dp(t)  =-b_x^{*\top}(t)p(t)dt-l_x^{*}(t)dt-\sigma_x^{*\top}(t)p(t)d^\circ B^H(t)+\sum_{j=1}^m  q_j(t)dW_j(t)\,, &\quad 0\le t\le T\\
\\
p(T)=h_x(x^*(T))&
\end{cases}
\end{equation}
Or
\[
p(T)=h_x(x^*(T))+\int_t^T \left[ b_x^{*\top}(s)p(s) -l_x^*(s)\right] ds+\int_t^T \sigma_x^{*\top}(s)p(s)d^\circ B^H(s)+ \int_t^T
q_j(s)dW_j(s)\,.
\]
Combining \eref{e.6.15}, \eref{e.6.16}, \eref{e.6.17},  and \eref{e.6.18}  we have
\begin{theorem}  Let the assumptions (H4) and (H5) be satisfied.
If $(u^*, x^*)$ is an optimal  pair of stochastic control
problems \eref{6.3.2}-\eref{e.5.4a}.
Then $(u^*, x^*)$  satisfies the following system of equations.
\begin{eqnarray}
\begin{cases}
dx(t)=b(t,x(t),u(t))dt+\sum\limits_{j=1}^m\sigma^j(t,x(t),u(t))d^\circ B_j^H(t)\,, \\
x(0)=x_0\,; \\ \\
dp(t)  =-b_x^{*\top}(t)p(t)dt-l_x^*(t)dt-\sigma_x^{*\top}(t)p(t)d^\circ B^H(t)+\sum_{j=1}^m  q_j(t)dW_j(t)\,, \\
p(T)=h_x(x^*(T))\,; \\ \\
\sum_{j=1}^m \si_u^{j, *\top } (t) p(t) =0\,; \\  \\
b_u^{*\top}(t)p(t)
+l_u^*(t)+ \int_0^t\phi_{1, H}(t-s)  \left( \frac{d}{ds} \int_s^t
(r-s)^{\frac12-H}r^{H-\frac12}      D_r ^j   \left(p(t)  \si_r^{j, *}(t)\right)
  dr\right)   ds\\
 \quad +   \si_u^{j, *}(t) \int_0^T\phi_{1, H}(t-s)  \left( \frac{d}{ds} \int_{t\vee s} ^T
(r-s)^{\frac12-H}r^{H-\frac12}      \EE \left[  D_r ^j  P(t)
\Big|\cF  _t\right] dr\right) ds=0\,,
\end{cases}
\end{eqnarray}
where
\[
P(t)= {\left(\Phi^{\top}\right) }^{ -1}(t) \left[\int_t^T\Phi^\top (s)l_x^*(s)ds+\Phi^\top (T)h_x(x^*(T)) \right]\,.
\]
\end{theorem}


\begin{thebibliography}{20}


\bibitem{BHOZ}{\sc    Biagini, F., Hu, Y., \O ksendal, B. and Zhang,
T.}\  {\em Stochastic calculus for fractional Brownian motion and
applications}. Springer,  2008.

\bibitem{BHOS}{\sc    Biagini, F., Hu, Y., \O ksendal, B. and Sulem, A.}\
{\em   A stochastic maximum principle for processes driven by
fractional Brownian motion}.
  Stochastic Process. Appl. 100 (2002),  233--253.

\bibitem{BSW}{\sc     Benes, V. E.; Shepp, L. A. and  Witsenhausen, H. S.}\ {\em
  Some solvable stochastic control problems}.
  Stochastics 4 (1980/81), no. 1, 39-83.

\bibitem{Bis1}
{\sc  Bismut, J. M.}\ {\em Conjugate convex functions in optimal
stochastic control}.  J. Math. Anal. Appl., 44 (1973),  384-404.

\bibitem{Bis2}
{\sc  Bismut, J. M.}\ {\em An introductory approach to duality in
optimal stochastic control}.  SIAM Rev., 20 (1978),  62-78.

\bibitem{Bis3}
{\sc  Bismut, J. M.}\ {\em  M\'ecanique al\'eatoire}.
Lecture Notes in Mathematics, 866. Springer, 1981.


\bibitem{CQ} {\sc  Coutin, L. and Qian, Z.} \
{\em  Stochastic analysis, rough path
analysis and fractional Brownian motions}.  Probab. Theory Related
Fields  \textbf{122} (2002), 108--140.

\bibitem{CS} {\sc
Cannarsa, P.  and  Sinestrari, C.} \ {\em
 Semiconcave functions, Hamilton-Jacobi equations, and optimal control}.
  Progress in Nonlinear Differential Equations and their Applications, 58.
  Birkh\"auser, 2004.

\bibitem{DM} {\sc Dellacherie}, C. and {\sc  Meyer}, P.-A.   Probabilities and potential.   North-Holland Publishing Co., Amsterdam-New York; North-Holland Publishing Co., Amsterdam-New York, 1978.

\bibitem{DHP}
{\sc Duncan, T., Hu, Y. and Pasik-Duncan, B.}\
 {\em Stochastic
calculus for fractional Brownian motion I. Theory}.  SIAM J. Control
Optim., 38(2) (2000),  582-612.

\bibitem{FS} {\sc
Fleming, W.  H.  and  Soner, H. M.}
\ {\em  Controlled Markov processes and viscosity solutions}. Second edition.
Stochastic Modelling and Applied Probability, 25.
Springer,
 2006.

\bibitem{FV} {\sc
Friz, P.  K. and  Victoir, N.  B.}\  {\em
Multidimensional stochastic processes as rough paths.
Theory and applications}. Cambridge Studies in Advanced Mathematics, 120.
Cambridge University Press, Cambridge, 2010.

\bibitem{Hau}
{\sc Haussmann, U. G.}\ {\em General necessary conditions for
optimal control of stochastic system}. Math. Programm. Stud., 6
(1976), 34--48.

\bibitem{Hu1}
{\sc Hu, Y.}\
 {\em Integral transformations and anticipative
calculus for fractional Brownian motions}. Memoirs of the American
mathematical society, Number 825, 2005.

\bibitem{Hu2}
{\sc Hu, Y.} \ {\em  Optimal consumption and portfolio in a market where the volatility
is driven by fractional Brownian motion}. Probability, Finance and
Insurance. Ed. Lai, T.L. et al. World Scientific Publishing.
164-173.

\bibitem{hustochastics} {\sc Hu, Y.} \ {\em  
Multiple integrals and expansion of  solutions  of  differential equations
driven by  rough paths and by fractional Brownian motions}.
To appear in   Stochastics   
 An International Journal of Probability and Stochastic Processes. 


\bibitem{HN1}
{\sc Hu,  Y. and Nualart, D.}\  {\em
Differential  equation  driven by
H\"older continuous functions of order greater than $1/2$}.
in  {\it The  Abel Symposium on Stochastic Analysis},
399-423. Springer, 2007.


\bibitem{HN2} {\sc Hu,  Y. and Nualart, D.}\
 {\em  Rough path analysis via fractional calculus}.
 Trans. Amer. Math. Soc. 361 (2009), no. 5, 2689--2718.


\bibitem{HNS}
{\sc Hu, Y., Nualart, D. and Song, X.}\ {\em
 Malliavin calculus for  backward
stochastic differential equations and application to numerical
schemes}.  The  Annals of Applied Probability 21 (2011), no. 6, 2379-2423.





\bibitem{HO} {\sc   Hu, Y. and  \O ksendal, B.}\
{\em  Partial information linear quadratic
 control for  jump diffusions}.
SIAM Journal of Control and Optimization, 47 (2008), 1744-1761.



\bibitem{HOS}{\sc   Hu, Y., \O ksendal, B. and Sulem, A.}\
{\em   Optimal consumption and portfolio in a Black-Scholes
 market driven by fractional Brownian motion}.
 Infin. Dimens. Anal. Quantum Probab. Relat. Top. 6 (2003), no. 4,  519-536.

\bibitem{HZ}{\sc    Hu, Y. and  Zhou, X.}\
{\em   Stochastic control for linear systems driven by fractional
noises}.
 SIAM J. Control Optim. 43 (2005),   2245-2277.

\bibitem{Kus}
{\sc Kushner, H. J. } {\em Necessary conditions for continuous
parameter stochastic optimization problems}. SIAM J. Control, 10
(1972),  550-565.

\bibitem{L} {\sc
Lions, P. L.}\
{\em  Generalized solutions of Hamilton-Jacobi equations}.
 Research Notes in Mathematics, 69. Pitman Advanced Publishing Program, 1982.

%

\bibitem{LQ}{\sc   Lyons,  T. and Qian,  Z.M.}\
{\em  System Control and Rough Paths}.
Clarendon Press. Oxford, 2002.

\bibitem{nisio} {\sc  Nisio, M. }\ {\em
Lectures on stochastic control theory}.
ISI Lecture Notes, 9. Macmillan Co. of India,   1981.


\bibitem{N}
{\sc Nualart, D.}\
 {\em The Malliavin Calculus and Related Topics}.
Springer,   2006.

\bibitem{NR}
{\sc Nualart, D. and Rascanu, S.}\
 {\em Differential equations driven
by fractional Brownian motion}.  Collect. Math., 53 (2002),  55-81.



\bibitem{P}
{\sc Peng, S.}\
 {\em A general stochastic maximum principle for
optimal control problems}. SIAM J. Control, 28 (1990),  966-979.


\bibitem{Yor}
{\sc Revuz, D.  and Yor, M.}\
 {\em Continuous Martingales and
Brownian Motion}.   Springer, 1994




\bibitem{SKM} {\sc Samko S. G., Kilbas A. A. and Marichev O. I}.  {\em
Fractional Integrals and Derivatives. Theory and Applications}.  Gordon and
Breach, 1993.



\bibitem{Won}
{\sc Wonham, W. M.}\ {\em On the separation theorem of stochastic
control}.   SIAM J. Control, 6 (1968),  312-326.


\bibitem{Yo} {\sc Young, L. C}.   {\em An inequality of the H\"{o}lder type connected
with Stieltjes integration}.  {Acta Math.}   {67} (1936) 251-282.

\bibitem{Za} {\sc Z\"{a}hle, M}.  {\em Integration with respect to fractal
functions and stochastic calculus. I} .  {Prob. Theory Relat. Fields}
 {111} (1998) 333-374.

\bibitem{YZ}
{\sc Yong, J. and Zhou, X.}\ {\em Stochastic Controls: Hamiltonian
Systems and HJB Equations}.   Springer, 1999.






\end{thebibliography}
\end{document}